\documentclass[letterpaper,11pt,oneside,reqno]{amsart}

\usepackage[sorting=nyt,style=alphabetic,backend=bibtex,hyperref=true,doi=false,maxbibnames=9,maxcitenames=4]{biblatex}
\makeatletter
\def\blx@maxline{77}
\makeatother
\usepackage{amsmath,amssymb,amsthm,amsfonts}
\usepackage{hyperref}
\usepackage{graphicx,color}
\usepackage{upgreek}
\usepackage[mathscr]{euscript}

\allowdisplaybreaks
\numberwithin{equation}{section}

\usepackage{tikz}
\usetikzlibrary{shapes,arrows,positioning,decorations.markings}

\usepackage{array}
\usepackage{adjustbox}
\usepackage{cleveref}
\usepackage{enumerate}

\usepackage[DIV12]{typearea}

\renewcommand{\Re}{\mathop{\mathrm{Re}}}
\renewcommand{\Im}{\mathop{\mathrm{Im}}}

\newcommand{\Det}{\mathop{\mathrm{det}}\limits}
\newcommand{\lozvdot}[2]
{
	\begin{scope}[shift={#1}]
		\draw [thick,fill=#2] (0,0) -- (.5,\rt) -- (1,0) -- (.5,-\rt) -- cycle;
		\draw[fill] (.5,0) circle (4pt);
	\end{scope}
}
\newcommand{\lozv}[2]
{
	\begin{scope}[shift={#1}]
		\draw [thick,fill=#2] (0,0) -- (.5,\rt) -- (1,0) -- (.5,-\rt) -- cycle;
	\end{scope}
}
\newcommand{\lozl}[2]
{
	\begin{scope}[shift={#1}]
		\draw [thick,fill=#2] (0,0) -- (-.5,\rt) -- (.5,\rt) -- (1,0) -- cycle;
	\end{scope}
}
\newcommand{\lozr}[2]
{
	\begin{scope}[shift={#1}]
		\draw [thick,fill=#2] (0,0) -- (.5,\rt) -- (1.5,\rt) -- (1,0) -- cycle;
	\end{scope}
}

\newtheorem{proposition}{Proposition}[section]
\newtheorem{lemma}[proposition]{Lemma}

\newtheorem{theorem}[proposition]{Theorem}
\newtheorem{conjecture}[proposition]{Conjecture}
\theoremstyle{definition}
\newtheorem{assumption}{Assumption}
\newtheorem{definition}[proposition]{Definition}
\newtheorem{remark}[proposition]{Remark}

\begin{document}

\title[GUE corners limit of $q$-distributed lozenge tilings]
{GUE corners limit of $q$-distributed lozenge tilings}
\author[S. Mkrtchyan]{Sevak Mkrtchyan}
\address{S. Mkrtchyan, Department of Mathematics, 
University of Rochester,
500 Joseph C. Wilson Blvd., Rochester, NY 14627}
\email{sevak.mkrtchyan@rochester.edu}

\author[L. Petrov]{Leonid Petrov}
\address{L. Petrov, Department of Mathematics, University of Virginia, 
141 Cabell Drive, Kerchof Hall,
P.O. Box 400137,
Charlottesville, VA 22904, USA,
\newline{}and Institute for Information Transmission Problems, Bolshoy Karetny per. 19, Moscow, 127994, Russia}
\email{lenia.petrov@gmail.com}
\date{}
\begin{abstract}
	We study asymptotics of $q$-distributed random lozenge tilings of sawtooth
	domains (equivalently, of random interlacing integer arrays with fixed top
	row). 
	Under the distribution we consider each tiling is weighted proportionally to
	$q^{\mathsf{vol}}$, where $\mathsf{vol}$ is the volume under the
	corresponding 3D stepped surface.
	We prove the following Interlacing Central Limit Theorem: as $q\rightarrow1$,
	the domain gets large, and the fixed top row approximates a given nonrandom
	profile, the vertical lozenges are distributed as the eigenvalues of a GUE
	random matrix and of its successive principal corners.
	Our results extend the GUE corners asymptotics for tilings of bounded
	polygonal domains previously known in the uniform (i.e., $q=1$) case.
	Even though $q$ goes to $1$, the presence of the $q$-weighting affects
	non-universal constants in our Central Limit Theorem.
\end{abstract}
\maketitle

\setcounter{tocdepth}{1}
\tableofcontents
\setcounter{tocdepth}{3}

\section{Introduction and main results}
\label{sec:intro}

\subsection{$q$-distributed lozenge tilings}

We begin with defining our main object, a probability distribution
$P_{q,\nu}^{N}$ on interlacing integer arrays
of depth $N$ with fixed top
row $\nu=(\nu_1\ge \ldots \ge \nu_N)$, $\nu_i\in\mathbb{Z}$.
Here $q>0$ is a parameter that will eventually tend to $1$.
By an interlacing array of depth $N$
we mean a collection 
\begin{equation*}
	\boldsymbol\lambda=\bigl\{ \lambda^{k}_j\in\mathbb{Z}\colon k=1,\ldots,N,\; j=1,\ldots,k  \bigr\}
\end{equation*}
satisfying the interlacing constraints 
$\lambda^{k}_j\le \lambda^{k-1}_{j-1}\le\lambda^{k}_{j-1}$
for all $k,j$.
Interlacing integer arrays originated as index sets of basis vectors in
irreducible representations of unitary groups, and consequently they are
sometimes referred to as Gelfand--Tsetlin schemes or patterns.
We will identify interlacing arrays with configurations of lozenges
of three types as shown in \Cref{fig:tiling_intro}.

\begin{figure}[htpb]
		\begin{tikzpicture}
			[scale=.45, thick]
			\def\rt{0.866025}
			\foreach \rl in
			{
				(-.5, 2*\rt),  (6.5, 2*\rt),  (12.5, 2*\rt),
				(4,   \rt),    (8,   \rt),    (11,   \rt), (12,\rt),
				(1.5, 0),      (2.5, 0),      (5.5,  0),      (10.5, 0), (11.5,0),
				(7,   -\rt),   (10,  -\rt), (11,-\rt),
				(4.5, -2*\rt), (8.5, -2*\rt), (9.5,  -2*\rt), (10.5,-2*\rt),
				(6,   -3*\rt), (7,   -3*\rt), (8,    -3*\rt), (9,    -3*\rt), (10,-3*\rt)
			}
			{
				\lozr{\rl}{red!20!white}
			}
			\foreach \ll in
			{
				(-1.5, 2*\rt),  (1.5, 2*\rt),  (2.5, 2*\rt),  (4.5, 2*\rt),  (5.5, 2*\rt),  (8.5, 2*\rt), (11.5,2*\rt),
				(-1,   \rt),    (0,   \rt),    (2,   \rt),    (3,   \rt),    (6,   \rt),    (7,   \rt),    (0, \rt),
				(-.5,  0),      (.5,  0),      (4.5, 0),      (7.5, 0),      (8.5, 0),
				(0,    -\rt),   (1,   -\rt),   (2,   -\rt),   (3,   -\rt),   (5,   -\rt),   (6,   -\rt),   (9, -\rt),
				(.5,   -2*\rt), (1.5, -2*\rt), (2.5, -2*\rt), (3.5, -2*\rt), (6.5, -2*\rt), (7.5, -2*\rt),
				(1,    -3*\rt), (2,   -3*\rt), (3,   -3*\rt), (4,   -3*\rt), (5,   -3*\rt)
			}
			{
				\lozl{\ll}{yellow!20!white}
			}
			\foreach \vl in
			{
				(-1,  3*\rt), (3,   3*\rt), (6,   3*\rt), (9,   3*\rt), (10,   3*\rt), (12, 3*\rt),
				(.5,  2*\rt), (3.5, 2*\rt), (7.5, 2*\rt), (9.5, 2*\rt), (10.5, 2*\rt),
				(1,   \rt),   (5,   \rt),   (9,   \rt),   (10,  \rt),
				(3.5, 0),     (6.5, 0),     (9.5, 0),
				(4,   -\rt),  (8,   -\rt),
				(5.5, -2*\rt)
			}
			{
				\lozvdot{\vl}{green!20!white}
			}
			\draw[->,opacity=.5] (-3,3*\rt)--++(18,0) node[right,opacity=1] 
			{6};
			\draw[->,opacity=.5] (-3,2*\rt)--++(18,0) node[right,opacity=1] 
			{5};
			\draw[->,opacity=.5] (-3,\rt)--++(18,0) node[right,opacity=1] 
			{4};
			\draw[->,opacity=.5] (-3,0)--++(18,0) node[right,opacity=1] 
			{3};
			\draw[->,opacity=.5] (-3,-\rt)--++(18,0) node[right,opacity=1] 
			{2};
			\draw[->,opacity=.5] (-3,-2*\rt)--++(18,0) node[right,opacity=1] 
			{1};
			\draw[->,opacity=.5] (4,-4*\rt)--++(-4.5,9*\rt);
		\end{tikzpicture}
	\caption{%
		Identification of an interlacing array with a configuration of lozenges. 
		To interpret the numbers $\lambda^k_j$ as distinct particles, we shift them to
		$\lambda^k_j+k-j$, and for each $k$ place 
		$\{\lambda^k_j+k-j\}_{1\le j\le k}$
		onto the $k$th level in the coordinate system as shown.
		Then we surround each particle by a vertical lozenge, and 
		complete the tiling with lozenges of two other types to get a 
		lozenge tiling of the corresponding sawtooth domain.
		For example, the top row here is a configuration 
		$(12,10,9,6,3,-1)$, so $\lambda^{6}=\nu=(7,6,6,4,2,-1)$.
		This tiling has
		$\mathsf{vol}(\boldsymbol\lambda)=56$.%
	}
	\label{fig:tiling_intro}
\end{figure}
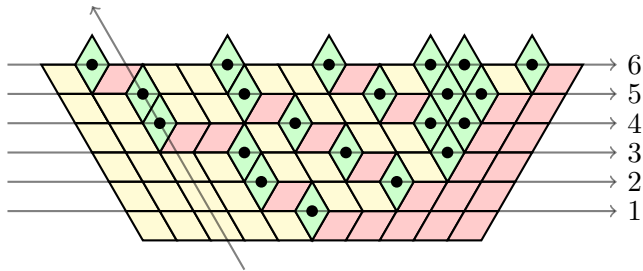

Define $|\lambda^k|:=\lambda^k_1+\ldots +\lambda^k_k$, and note that 
this number is not necessarily nonnegative. 
We will stick to the convention that the volume of a configuration
$\boldsymbol\lambda$ is 
\begin{equation*}
	\mathsf{vol}(\boldsymbol\lambda):=
	\sum_{k=1}^{N-1}|\lambda^k|.
\end{equation*}
One can see that 
for fixed $N$ and $\lambda^N$, 
up to an additive constant depending on $\lambda^N$, 
$\mathsf{vol}(\boldsymbol\lambda)$ can be interpreted 
as a volume behind (i.e., to the left of) the 3-dimensional surface as in \Cref{fig:tiling_intro}
under the agreement that the normal vector to the vertical lozenges points to the right.

Let $P_{q,\nu}^{N}$ be the probability distribution on the finite set of 
interlacing arrays $\boldsymbol\lambda$ of depth $N$ and 
with fixed top row $\lambda^N=\nu$ defined as
\begin{equation*}
	P_{q,\nu}^N(\boldsymbol\lambda):=\frac{q^{\mathsf{vol}(\boldsymbol\lambda)}}{Z_{z,\nu}^N},
\end{equation*}
where $Z_{q,\nu}^N$ is the normalization constant.
We have 
\begin{equation}
	\label{Schur_q_specialization}
	Z_{q,\nu}^N=s_\nu(1,q,\ldots ,q^{N-1})=\prod_{1\le i<j\le N}\frac{q^{\nu_i-i}-q^{\nu_j-j}}{q^{-i}-q^{-j}}, 
\end{equation}
where $s_\nu$ is the Schur polynomial
(see, e.g., \cite{Petrov2012GT} for details).

\subsection{Asymptotic regime}

Let us now describe the limit regime we consider.
Define
\begin{equation*}
	f_N(x):=\frac{\nu_{\lceil Nx \rceil}}{N},\qquad x\in [0,1],
\end{equation*}
see \Cref{fig:function_N_intro}.
This is a weakly decreasing function.

\begin{figure}[htpb]
	\begin{tikzpicture}
		[scale=1.5, thick]
		\def\sixth{.166666667}
		\def\koef{2}
		\draw[->] (-.2,0)--++(1.4*\koef,0) node [right] {$x$};
		\draw[->] (0,-.4)--++(0,1.9) node [left] {$f_N(x)$};
		\draw[fill] (0,0) circle(.9pt) node [below left] {0};
		\draw (1*\koef,-.05)--++(0,.1) node[above] {1};
		\draw (.05,1)--++(-.1,0) node[left] {1};
		\draw[ultra thick,gray] (0,1+\sixth)--++(\sixth*\koef,0)--++(0,-\sixth)--++(2*\sixth*\koef,0)
		--++(0,-2*\sixth)--++(\sixth*\koef,0)--++(0,-2*\sixth)--++(\sixth*\koef,0)--++(0,-3*\sixth)--++(\sixth*\koef,0);
	\end{tikzpicture}
	\caption{Example of the function $f_N(x)$ for $N=6$ and $\nu$ as in \Cref{fig:tiling_intro}.}
	\label{fig:function_N_intro}
\end{figure}
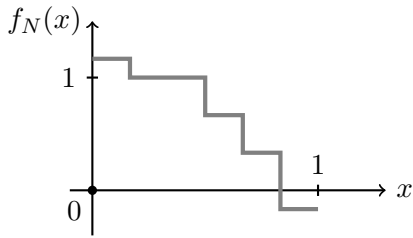

The depth $N$ of the interlacing array 
will play the role of the main parameter going to infinity.
Our main asymptotic assumptions are the following.

\begin{assumption}
	\label{ass:function}
	The rescaled top rows $f_N(x)$ converge to a weakly decreasing,
	nonconstant,
	piecewise continuous
	function $\mathsf{f}(x)$ in the sense that
	\begin{equation*}
		\sup_{x\in[0,1]}
		\left|
			f_N(x)-\mathsf{f}(x)
			\right|=o(N^{-\frac{1}{2}}),\qquad N\to+\infty.
	\end{equation*}
	Without loss of generality we can and will assume that 
	$\mathsf{f}(1)=0$.
\end{assumption}

\begin{remark}
	\label{rmk:hexagon}
	The simplest nontrivial case of the function $\mathsf{f}$ is 
	\begin{equation*}
		\mathsf{f}(x)=
		\begin{cases}
			a,&0\le x\le b;\\
			0,&b<x\le1,
		\end{cases}
	\end{equation*}
	where $0<b<1$ and $a>0$. 
	Then the measures $P^{N}_{q,\nu}$ are supported
	by lozenge tilings of growing lattice hexagons, 
	and after the rescaling the lattice hexagons approximate
	the hexagon with sides $a$, $b$, $1-b$, $a$, $b$, $1-b$, cf.
	\Cref{fig:hexagon}.

	\begin{figure}[htpb]
		\begin{tikzpicture}
			[scale=.3, thick]
			\def\rt{0.866025}
			\foreach \vl in
			{
				(0,10*\rt),(1,10*\rt),(2,10*\rt),(3,10*\rt),(4,10*\rt),(5,10*\rt),
				(15,10*\rt),(16,10*\rt),(17,10*\rt),(18,10*\rt),(19,10*\rt),(20,10*\rt)
				,
				(0.5,9*\rt),(1.5,9*\rt),(2.5,9*\rt),(3.5,9*\rt),(4.5,9*\rt),
				(15.5,9*\rt),(16.5,9*\rt),(17.5,9*\rt),(18.5,9*\rt),(19.5,9*\rt)
				,
				(1,8*\rt),(2,8*\rt),(3,8*\rt),(4,8*\rt),
				(16,8*\rt),(17,8*\rt),(18,8*\rt),(19,8*\rt)
				,
				(1.5,7*\rt),(2.5,7*\rt),(3.5,7*\rt),
				(16.5,7*\rt),(17.5,7*\rt),(18.5,7*\rt)
				,
				(2,6*\rt),(3,6*\rt),
				(17,6*\rt),(18,6*\rt)
				,
				(2.5,5*\rt),
				(17.5,5*\rt)
			}
			{
				\lozv{\vl}{green!20!white}
			}
			\draw (6,10*\rt)--(15,10*\rt);
			\draw (18,4*\rt)--(15,-2*\rt)--(6,-2*\rt)--(3,4*\rt);
			\node at (14.4,7*\rt) {$\lfloor bN \rfloor $};
			\node at (9,11.2*\rt) {$\lfloor aN \rfloor $};
			\node at (19.8,0*\rt) {$\lfloor (1-b)N \rfloor $};
		\end{tikzpicture}
		\caption{An example of a lattice hexagon with $N=12$.}
		\label{fig:hexagon}
	\end{figure}
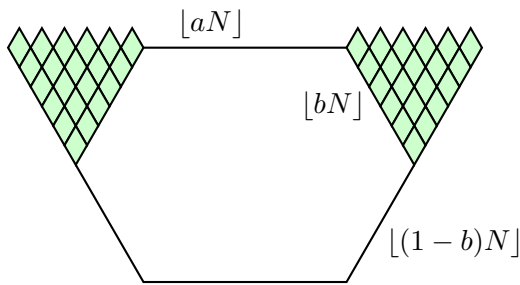

	In general, if $\mathsf{f}(s)$ is piecewise constant then the 
	measures $P^{N}_{q,\nu}$ are supported on lozenge tilings of 
	growing lattice polygons with fixed number of sides 
	each of which grows linearly in $N$.
\end{remark}

\begin{assumption}
	\label{ass:q}
	The parameter $q\in(0,1)$ depends on $N$ and converges to $1$ as
	$q=q(N):=e^{-\upgamma/N}$, where $\upgamma\in(0,+\infty)$ is fixed.
\end{assumption}

In \Cref{sub:slower_q} below we discuss other regimes of $q\nearrow1$
besides the one of Assumption \ref{ass:q}.

\subsection{The GUE corners process}

Let us now describe the limiting object entering our Central Limit Theorem:

\begin{definition}(GUE eigenvalue distribution and GUE corners process)
	\label{def:GUE}
	Consider a $K\times K$ random matrix $H$ from the Gaussian Unitary Ensemble (GUE)
	with variance $\upsigma^2>0$. That is, 
	$H=[H_{ij}]_{i,j=1}^K$, $H^*=H$, 
	and $\Re H_{ij}\sim \mathsf{N}\bigl(0,\frac{1+\mathbf{1}_{i=j}}{2}\,\upsigma^2\bigr)$,
	$i\ge j$,
	while $\Im H_{ij}\sim \mathsf{N}\bigl(0,\frac{1}{2}\upsigma^2\bigr)$, $i>j$.
	Let $\mathsf{L}^K=(\mathsf{L}^K_1\ge \ldots\ge \mathsf{L}^K_K)\in\mathbb{R}^K$ 
	denote the ordered eigenvalues of $H$.
	Their joint probability density has the form
	(e.g., see \cite{mehta2004random} or \cite{AndersonGuionnetZeitouniBook})
	\begin{equation}
		\label{GUE_joint_density}
		\mathrm{Prob}(\mathsf{L}^K_1\in d\xi_1,\ldots,\mathsf{L}^K_K\in d\xi_K )
		=
		\frac{1}{0!1!\ldots(K-1)!\upsigma^{K(K-1)}}
		\prod_{1\le i<j\le K}(\xi_i-\xi_j)^2
		\prod_{j=1}^{K}
		\frac{e^{-\xi_j^2/(2\upsigma^2)}}{\sqrt{2\pi \upsigma^2}}d\xi_j
		.
	\end{equation}
	We will denote the distribution of $\mathsf{L}^K$ by 
	$\mathsf{GUE}_K(\upsigma^2)$, and this will be referred to as the 
	GUE eigenvalue distribution.

	Along with the eigenvalues $\mathsf{L}^K$ of $H$, consider 
	eigenvalues $\mathsf{L}^r=(\mathsf{L}^r_1\ge \ldots\ge \mathsf{L}^r_r )\in\mathbb{R}^{r}$
	of the principal $r$-corner $[H_{ij}]_{i,j=1}^r$ of $H$ for each $r=1,\ldots ,K $.
	These eigenvalues satisfy the interlacing constraints
	\begin{equation*}
		\mathsf{L}^r_j\le \mathsf{L}^{r-1}_{j-1}\le \mathsf{L}^r_{j-1},
		\qquad r=1,\ldots,K,\quad
		j=2,\ldots,r .
	\end{equation*}
	The marginal distribution of each $\mathsf{L}^r$ is the same as \eqref{GUE_joint_density}
	(with $K$ replaced by $r$). 
	Note that $\mathsf{L}^{1}_{1}$ is simply a normal 
	random variable with mean $0$ and variance $\upsigma^2$.
	The joint distribution of the whole interlacing array
	$\mathbf{L}=\{\mathsf{L}^r_j\colon 1\le r\le K,\,1\le j\le r \}\in \mathbb{R}^{\frac{K(K+1)}{2}}$
	will be referred to as the GUE corners process. We will denote it by
	$\mathsf{GUE}_{K\times K}(\upsigma^2)$.
\end{definition}

The GUE corners process (earlier also called the GUE minors process)
is a classical object of random matrix theory.
It appeared in the context of interacting particle systems 
in \cite{Baryshnikov_GUE2001}
and also in connection with random domino and lozenge tilings in 
\cite{Johansson2005arctic},
\cite{OkounkovReshetikhin2006RandomMatr},
and
\cite{johansson2006eigenvalues}.
The latter three works contain formulas for the correlation kernel
of the GUE corners process (which is a determinantal point process),
but we do not need the kernel in the present paper.
See also \Cref{sub:previous} below for an exposition of
previous results on GUE corners limits in random lozenge tilings.

\subsection{Interlacing Central Limit Theorem}

To formulate our main result let us introduce 
the limiting global location
\begin{equation}
	\mathsf{u}
	:=
	\frac{1}{\upgamma}
	\log
	\left( 
		\frac{\int_0^1 e^{\upgamma s}ds}
		{\int_0^1 e^{\upgamma(s-\mathsf{f}(s))}ds}
	\right)
	\label{u_intro}
\end{equation}
and the limiting variance
\begin{equation}
	\upsigma^2:=
	\frac{1}{(e^{\upgamma}-1)^{2}}
	\int_0^1
	\left(
		e^{2\upgamma(\mathsf{u}+s-\mathsf{f}(s))}
		-
		e^{2s\upgamma}
	\right)ds
	\label{sigma_square}
\end{equation}
depending on $\upgamma>0$ and the limiting profile 
$\mathsf{f}(s)$
of the $N$th row of the interlacing array.

\begin{theorem}
	\label{thm:main_result}
	Fix any $K\ge1$.
	Under Assumptions \ref{ass:function} and \ref{ass:q}, 
	we have the following convergence in distribution
	on $\mathbb{R}^{\frac{K(K+1)}{2}}$ as $N\to+\infty$:
	\begin{equation*}
		\biggl\{ \frac{\lambda^r_j(N)-\mathsf{u}N}{\sqrt N} \colon 1\le r\le K,\,1\le j\le r \biggr\}
		\to
		\left\{ \mathsf{L}^r_j\colon 1\le r\le K,\,1\le j\le r \right\}\sim
		\mathsf{GUE}_{K\times K}(\upsigma^2),
	\end{equation*}
	where $\lambda^r_j=\lambda^r_j(N)$ denote the locations of the vertical lozenges
	on the first $K$ levels under the distribution $P^N_{q,\nu}$.
\end{theorem}

We call \Cref{thm:main_result} on Gaussian asymptotics of 
vertical lozenges at finitely many bottommost rows of the tiling
the Interlacing Central Limit Theorem (ICLT).
The proof of \Cref{thm:main_result} occupies
\Cref{sec:prelimit_formula,sec:asymptotic_mean,sec:critical_points,sec:completing_proof}.
Let us make two remarks in connection with this theorem.

\subsubsection{Frozen boundary}

Besides \Cref{thm:main_result} describing the asymptotic behavior 
of finitely many rows of the interlacing array adjacent to the bottom
boundary (cf. \Cref{fig:tiling_intro})
there are several other interesting 
limit regimes in various ensembles of random lozenge tilings. 
Let us only discuss the connection of ICLT and
the 
law of large numbers established
in \cite{CohnLarsenPropp}, \cite{CohnKenyonPropp2000}, and
\cite{OkounkovKenyon2007Limit}
both in the uniform and in the $q^{\mathsf{vol}}$-weighted case. 
This law of large numbers implies the existence of a frozen boundary curve
separating the liquid phase (in which asymptotically all types of lozenges are present)
from the frozen regions (consisting of lozenges of only one type),
see \Cref{fig:simulation}.
The ICLT (as in \Cref{thm:main_result}) arises 
at a point where two frozen regions adjacent to the bottom boundary meet.
Thus, the global location $\mathsf{u}$ \eqref{u_intro} corresponds to the 
point where the frozen boundary curve is tangent to the bottom boundary.

To the best knowledge of the authors, the frozen boundary 
for $q^{\mathsf{vol}}$-weighted lozenge tilings is explicitly 
known
only in the case of unbounded 
plane partitions \cite{okounkov2003correlation}, \cite{Okounkov2005} 
and for the 
hexagon \cite{borodin-gr2009q}.\footnote{%
	In fact, this paper
	considers a more general weighting of lozenge tilings related to the 
	$q$-Racah univariate orthogonal polynomials. 
	In the uniform ($q=1$) case 
	the frozen boundary is known explicitly for a much wider family of boundary conditions,
	e.g., see
	\cite{Petrov2012GT},
	\cite{GorinBufetov2013free},
	\cite{duse2015asymptotic}.

	We also note that
	\cite{OkounkovKenyon2007Limit}
	implies that for a bounded polygon (both in the uniform and in the $q^{\mathsf{vol}}$-weighted case)
	the frozen boundary solves an algebraic equation
	which can be found approximately via numeric homotopy.
	By explicitly knowing the frozen boundary we mean an explicit solution
	to such an equation in the form of, e.g., a parametrization 
	like in 
	\cite{Petrov2012GT} for sawtooth domains in the uniform case.%
}
One can check that the tangent point of the frozen boundary from \cite{borodin-gr2009q}
and the boundary of the hexagon
coincides with $\mathsf{u}$.

At generic points of the frozen boundary
the locations of the vertical lozenges
form a determinantal point process 
with the extended Airy kernel in the scaling
limit \cite{Okounkov2005}, \cite{DusseMetcalfeAiry2017}.

\begin{figure}[htpb]
	\hspace*{-0.15\textwidth}\begin{tikzpicture}
		[scale=1]
		\node[anchor=south] at (0,0) {\includegraphics[width=1.3\textwidth]{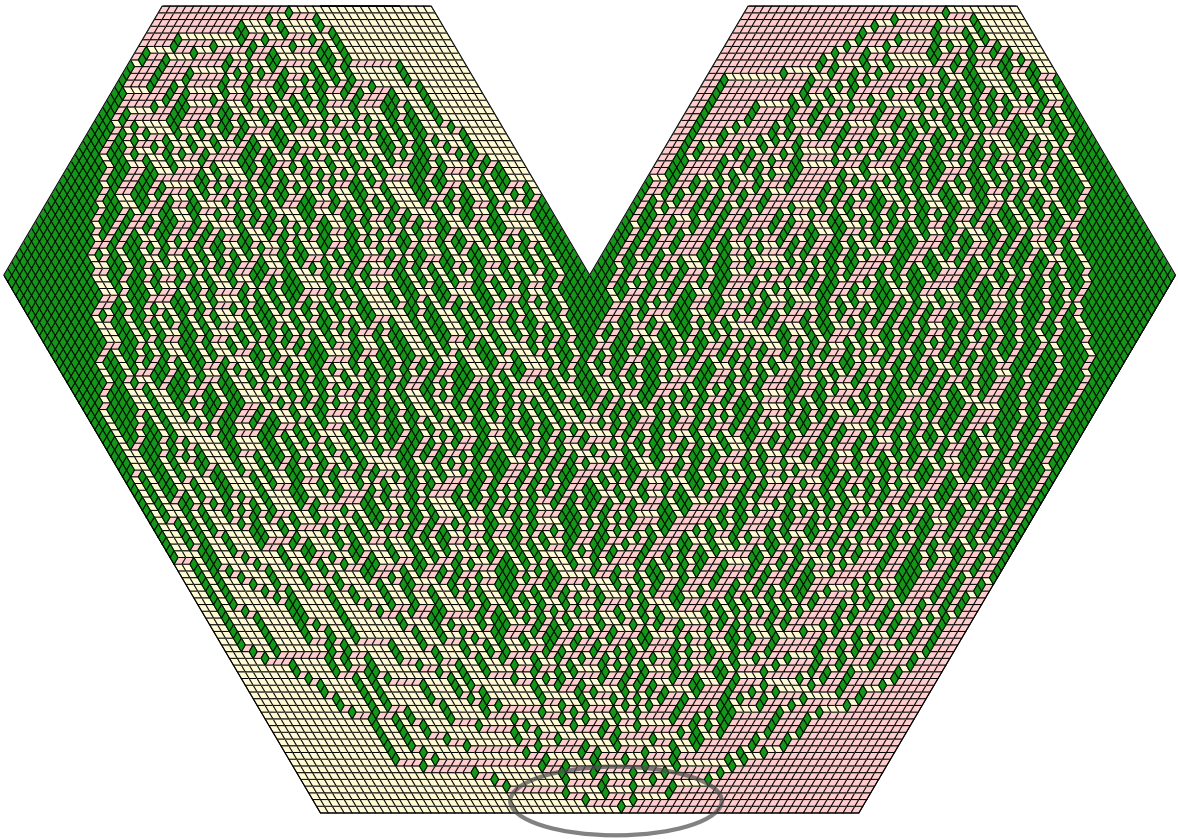}};
		\draw[ultra thick, color=black!70!white, opacity=.7] (0.35,0.7) ellipse (40pt and 13pt);
	\end{tikzpicture}
	\caption{%
		A uniformly random lozenge tiling of a 9-gon. The frozen boundary is a cardioid.
		A neighborhood of the turning point where the frozen boundary is tangent to the bottom
		boundary of the polygon is highlighted.
	}
	\label{fig:simulation}
\end{figure}

\subsubsection{Extension to all $\upgamma\in\mathbb{R}$}

Sending $\upgamma\to0$ in \eqref{u_intro} and \eqref{sigma_square}
recovers the parameters in the ICLT
for the uniform case
\cite{GorinPanova2012_full} (see also \cite{novak2015lozenge}):

\begin{proposition}[Uniform case $\upgamma=0$]
	\label{prop:q=1}
	We have
	\begin{equation*}
		\lim_{\upgamma\searrow0}\mathsf{u}
		=
		\int_0^1\mathsf{f}(s)ds
		,
		\qquad 
		\lim_{\upgamma\searrow0}\upsigma^2
		=
		\int_0^1 \mathsf{f}(s)^2ds
		-
		\Bigl( \int_0^1\mathsf{f}(s)ds \Bigr)^2
		+
		\int_0^1(1-2s)\mathsf{f}(s)ds
		.
	\end{equation*}
\end{proposition}
\begin{proof}
	For shorter notation, let $\bar g$ denote the integral of $g(s)$ over $s\in[0,1]$.
	The first claim follows from a Taylor expansion:
	\begin{align*}
		\upgamma\mathsf{u}
		&=
		\log\int_0^1 
		\bigl(
			1+\upgamma s+\tfrac12\upgamma^2s^2+O(\upgamma^3)
		\bigr)ds
		-
		\log\int_0^1 
		\bigl(
			1+\upgamma(s-\mathsf{f}(s))+\tfrac12\upgamma^2(s-\mathsf{f}(s))^2+O(\upgamma^3)
		\bigr)ds
		\\&=
		\log
		\left( 1+\tfrac12 \upgamma+\tfrac 16\upgamma^2+O(\upgamma^3) \right)
		-
		\log
		\left( 
			1+\tfrac12\upgamma- \upgamma
			\bar{\mathsf{f}}
			+
			\tfrac16 \upgamma^2+\tfrac12\upgamma^2\overline{\mathsf{f}^2}
			-\upgamma^2 \overline{s\mathsf{f}}
			+O(\upgamma^3)
		\right)
		\\&=
		\upgamma \bar{\mathsf{f}}+
		\tfrac12\upgamma^2\left( (\bar{\mathsf{f}})^2-\bar{\mathsf{f}}-\overline{\mathsf{f}^2}+2\overline{s\mathsf{f}} \right)
		+O(\upgamma^3).
	\end{align*}
	This implies the first claim, and we will also need the next term in $\mathsf{u}$ to address the 
	second claim.
	Namely, we have
	\begin{align*}
		(\upgamma^2+O(\upgamma^3))\upsigma^2
		&=
		\int_0^1
		\left(
			1+2\upgamma(\mathsf{u}+s-\mathsf{f}(s))+2\upgamma^2(\mathsf{u}+s-\mathsf{f}(s))^2
			-1-2s\upgamma-2s^2\upgamma^2+O(\upgamma^3)
		\right)ds
		\\&=
		2\upgamma(\mathsf{u}-\bar{\mathsf{f}})
		+2\upgamma^2
		\bigl( \mathsf{u}^2+\overline{\mathsf{f}^2}+\mathsf{u}-2\overline{s \mathsf{f}}-2\mathsf{u}\bar{\mathsf{f}} \bigr)
		+O(\upgamma^3)
		\\&=
		\upgamma^2
		\bigl( (\bar{\mathsf{f}})^2-\bar{\mathsf{f}}-\overline{\mathsf{f}^2}+2\overline{s\mathsf{f}} \bigr)
		+
		2\upgamma^2
		\bigl( -(\bar{\mathsf{f}})^2+\overline{\mathsf{f}^2}+\bar{\mathsf{f}}
		-2\overline{s \mathsf{f}} \bigr)
		+O(\upgamma^3),
	\end{align*}
	which implies the second claim.
\end{proof}

The next statement allows us to extend \Cref{thm:main_result}
to the case $q>1$, $q\searrow1$, that is, to $\upgamma<0$.
Observe that 
by reflecting the interlacing array with respect to zero
and shifting it one can turn the measure
$P^{N}_{q,\nu}$, where $q=e^{-\upgamma/N}<1$, $\upgamma>0$, 
into the measure $P^{N}_{1/q,\hat \nu}$, where
$\hat\nu_j=\nu_1-\nu_{N+1-j}$ for $j=1,\ldots,N$.
If $\nu$ corresponds to a function $\mathsf{f}(s)$ as in Assumption \ref{ass:function},
then $\hat\nu$ corresponds to the function $\hat{\mathsf{f}}(s)=\mathsf{f}(0)-\mathsf{f}(1-s)$.
Clearly, if \Cref{thm:main_result} holds for a sequence of measures of the form $P^{N}_{q,\nu}$, then 
it also holds for the measures $P^{N}_{1/q,\hat\nu}$, and the next proposition 
matches the limiting parameters:

\begin{proposition}[Symmetry]
	\label{prop:symmetry}
	Let $\mathsf{u},\upsigma^2$ correspond to the data $(\upgamma,\mathsf{f})$, where $\upgamma>0$,
	via \eqref{u_intro}, \eqref{sigma_square}.
	Then the data $(-\upgamma,\hat{\mathsf{f}})$
	leads to the parameters $\hat{\mathsf{u}}=\mathsf{f}(0)-\mathsf{u}$ and
	$\hat\upsigma^2=\upsigma^2$.
\end{proposition}
\begin{proof}
	A straightforward verification.
\end{proof}

\begin{remark}
	\label{rmk:nonpositive_gamma}
	\Cref{prop:q=1,prop:symmetry} imply that 
	\Cref{thm:main_result} holds 
	in the regime $q=e^{-\upgamma/N}$ for any fixed $\upgamma\in \mathbb{R}$.
	For simplicity in the rest of the paper we will continue to assume that $\upgamma$ is positive.
\end{remark}

\subsection{Slower $q\nearrow1$ regime}
\label{sub:slower_q}

\Cref{thm:main_result}
establishes,
in the regime 
$q=e^{-\upgamma/N}\nearrow1$ with $\upgamma>0$ fixed,
a universal ICLT 
regardless of the
(nonconstant) function $\mathsf{f}(s)$
governing the
asymptotic behavior of the
top row.
Let us argue that when $q=q(N)$ converges to $1$ slower, 
the presence of the GUE corners limit is not guaranteed and 
depends on the nature of the function $\mathsf{f}(s)$.

The law of large numbers
\cite{CohnLarsenPropp}, \cite{CohnKenyonPropp2000},
\cite{OkounkovKenyon2007Limit} implies the existence of 
a nonempty liquid region at scale $N$ 
for any fixed $\upgamma$ in the regime $q=e^{-\upgamma/N}$.
The slower $q\nearrow1$ regime leads 
to a completely frozen situation at scale $N$, 
i.e., there is no liquid region, 
and the random tiling at this scale looks like the unique tiling of the
minimal volume (see \Cref{fig:min_volume} for examples of the latter).
However, one can still ask if there is an ICLT at some fluctuation scale $\ll N$.

The slower convergence $q\nearrow1$ is equivalent to 
taking $\upgamma=\upgamma(N)$ depending on $N$ such that 
$\upgamma(N)\to+\infty$ but $\upgamma(N)\ll N$.
Let us first formally send $\upgamma\to+\infty$
in the limiting quantities \eqref{u_intro}, \eqref{sigma_square}
in the following
two cases of the top row:
\begin{equation*}
	\mathsf{f}_1(s):=\begin{cases}
		1,&0\le s\le\frac{1}{2};\\
		0,&\frac{1}{2}<s\le 1,
	\end{cases}
	\qquad\qquad
	\mathsf{f}_2(s)=1-s.
\end{equation*}
The function $\mathsf{f}_1$ corresponds to the hexagon boundary conditions,
and $\mathsf{f}_2$ --- to a top row having limiting density $\frac{1}{2}$
(the latter can be modeled by taking $\nu_j=N-j$).
Straightforward computations show that as $\upgamma\to+\infty$:
\begin{equation*}
	\mathsf{u}_1\sim \frac{e^{-\upgamma/2}}{\upgamma},\qquad 
	\upsigma_1^2\sim \frac{e^{-\upgamma/2}}{\upgamma},
	\qquad \qquad 
	\mathsf{u}_2\sim \frac{\log 2}{\upgamma}
	,\qquad 
	\upsigma_2^2\sim \frac{1}{2\upgamma}
	.
\end{equation*}
This suggests that if there could be an ICLT in the slower $q\nearrow1$ regime
then its scale must depend on the form of the function $\mathsf{f}(s)$.

In the hexagon case
there are, however, strong indications that an
ICLT at a finite distance from the bottom row
is not possible. 
Namely, in the slower $q\nearrow1$ regime 
the nontrivial asymptotic behavior (at a scale $\ll N$)
occurs in a neighborhood of the unique three-lozenge combination
\begin{tikzpicture}
	[scale=.22, thick]
	\def\rt{0.866025}
	\lozr{(.5,-\rt)}{red!20!white}
	\lozl{(1,0)}{yellow!20!white}
	\lozv{(0,0)}{green!20!white}
\end{tikzpicture}
in the minimal volume tiling
(cf. \Cref{fig:min_volume}, left).
Observe that the distance of
\begin{tikzpicture}
	[scale=.22, thick]
	\def\rt{0.866025}
	\lozr{(.5,-\rt)}{red!20!white}
	\lozl{(1,0)}{yellow!20!white}
	\lozv{(0,0)}{green!20!white}
\end{tikzpicture}
from the bottom row is of order $N$.
The configuration around 
\begin{tikzpicture}
	[scale=.22, thick]
	\def\rt{0.866025}
	\lozr{(.5,-\rt)}{red!20!white}
	\lozl{(1,0)}{yellow!20!white}
	\lozv{(0,0)}{green!20!white}
\end{tikzpicture}
should asymptotically coincide with the $q^{\mathsf{vol}}$-weighted random plane 
partition without boundary conditions 
(the latter model was studied in, e.g., \cite{ferrari2003step}, \cite{okounkov2003correlation}, \cite{Okounkov2005}).
The behavior of the bottommost lozenge of the type
\begin{tikzpicture}
	[scale=.22, thick]
	\def\rt{0.866025}
	\lozl{(1,0)}{yellow!20!white}
\end{tikzpicture}
can then be guessed from the corresponding results on random plane partitions
\cite{mutafchiev2006size}, \cite{Vershik2006}.
In particular, 
the fluctuations of the row number containing the bottommost lozenge
\begin{tikzpicture}
	[scale=.22, thick]
	\def\rt{0.866025}
	\lozl{(1,0)}{yellow!20!white}
\end{tikzpicture}
should grow to infinity.
Thus, even if we tune the speed at which $q$ goes to $1$ so that 
the plane partition behavior is visible (in any sense) at a finite distance from the bottom, 
still with positive probability the vertical lozenges 
at the bottommost rows are packed to the left (as in the minimal volume tiling).
This suggests that for the hexagon (and similarly for any polygonal boundary conditions)
the ICLT at a finite distance from the bottom does not occur.

\begin{figure}[htpb]
		\begin{tikzpicture}
			[scale=.3, thick]
			\def\rt{0.866025}
			\foreach \rl in
			{
				(15.5,3*\rt),(16.5,3*\rt),(8.5,3*\rt),(9.5,3*\rt),(10.5,3*\rt),(11.5,3*\rt),(12.5,3*\rt),(13.5,3*\rt),(14.5,3*\rt)
				,
				(15,2*\rt),(16,2*\rt),(8,2*\rt),(9,2*\rt),(10,2*\rt),(11,2*\rt),(12,2*\rt),(13,2*\rt),(14,2*\rt)
				,
				(15.5,1*\rt),(7.5,1*\rt),(8.5,1*\rt),(9.5,1*\rt),(10.5,1*\rt),(11.5,1*\rt),(12.5,1*\rt),(13.5,1*\rt),(14.5,1*\rt)
				,
				(15,0*\rt),(7,0*\rt),(8,0*\rt),(9,0*\rt),(10,0*\rt),(11,0*\rt),(12,0*\rt),(13,0*\rt),(14,0*\rt)
				,
				(6.5,-1*\rt),(7.5,-1*\rt),(8.5,-1*\rt),(9.5,-1*\rt),(10.5,-1*\rt),
				(11.5,-1*\rt),(12.5,-1*\rt),(13.5,-1*\rt),(14.5,-1*\rt)
				,
				(6,-2*\rt),(7,-2*\rt),(8,-2*\rt),(9,-2*\rt),(10,-2*\rt),(11,-2*\rt),(12,-2*\rt),(13,-2*\rt),(14,-2*\rt)
			}
			{
				\lozr{\rl}{red!20!white}
			}
			\foreach \ll in
			{
				(6.5,9*\rt),(7.5,9*\rt),(8.5,9*\rt),(9.5,9*\rt),(10.5,9*\rt),(11.5,9*\rt),(12.5,9*\rt),(13.5,9*\rt),(14.5,9*\rt)
				,
				(15,8*\rt),(7,8*\rt),(8,8*\rt),(9,8*\rt),(10,8*\rt),(11,8*\rt),(12,8*\rt),(13,8*\rt),(14,8*\rt)
				,
				(15.5,7*\rt),(7.5,7*\rt),(8.5,7*\rt),(9.5,7*\rt),(10.5,7*\rt),(11.5,7*\rt),(12.5,7*\rt),(13.5,7*\rt),(14.5,7*\rt)
				,
				(15,6*\rt),(16,6*\rt),(8,6*\rt),(9,6*\rt),(10,6*\rt),(11,6*\rt),(12,6*\rt),(13,6*\rt),(14,6*\rt)
				,
				(15.5,5*\rt),(16.5,5*\rt),(8.5,5*\rt),(9.5,5*\rt),(10.5,5*\rt),(11.5,5*\rt),(12.5,5*\rt),(13.5,5*\rt),(14.5,5*\rt)
				,
				(15,4*\rt),(16,4*\rt),(17,4*\rt),(9,4*\rt),(10,4*\rt),(11,4*\rt),(12,4*\rt),(13,4*\rt),(14,4*\rt)
			}
			{
				\lozl{\ll}{yellow!20!white}
			}
			\foreach \vl in
			{
				(0,10*\rt),(1,10*\rt),(2,10*\rt),(3,10*\rt),(4,10*\rt),(5,10*\rt),
				(15,10*\rt),(16,10*\rt),(17,10*\rt),(18,10*\rt),(19,10*\rt),(20,10*\rt)
				,
				(0.5,9*\rt),(1.5,9*\rt),(2.5,9*\rt),(3.5,9*\rt),(4.5,9*\rt),(5.5,9*\rt),
				(15.5,9*\rt),(16.5,9*\rt),(17.5,9*\rt),(18.5,9*\rt),(19.5,9*\rt)
				,
				(1,8*\rt),(2,8*\rt),(3,8*\rt),(4,8*\rt),(5,8*\rt),(6,8*\rt),
				(16,8*\rt),(17,8*\rt),(18,8*\rt),(19,8*\rt)
				,
				(1.5,7*\rt),(2.5,7*\rt),(3.5,7*\rt),(4.5,7*\rt),(5.5,7*\rt),(6.5,7*\rt),
				(16.5,7*\rt),(17.5,7*\rt),(18.5,7*\rt)
				,
				(2,6*\rt),(3,6*\rt),(4,6*\rt),(5,6*\rt),(6,6*\rt),(7,6*\rt),
				(17,6*\rt),(18,6*\rt)
				,
				(2.5,5*\rt),(3.5,5*\rt),(4.5,5*\rt),(5.5,5*\rt),(6.5,5*\rt),(7.5,5*\rt),
				(17.5,5*\rt)
				,
				(3,4*\rt),(4,4*\rt),(5,4*\rt),(6,4*\rt),(7,4*\rt),(8,4*\rt)
				,
				(3.5,3*\rt),(4.5,3*\rt),(5.5,3*\rt),(6.5,3*\rt),(7.5,3*\rt)
				,
				(4,2*\rt),(5,2*\rt),(6,2*\rt),(7,2*\rt)
				,
				(4.5,1*\rt),(5.5,1*\rt),(6.5,1*\rt)
				,
				(5,0*\rt),(6,0*\rt)
				,
				(5.5,-\rt)
			}
			{
				\lozv{\vl}{green!20!white}
			}
			\begin{scope}[shift={(28,0)}]
				\foreach \rl in
				{
					(21,8*\rt)
					,
					(20.5,7*\rt),(19.5,7*\rt)
					,
					(18,6*\rt),(19,6*\rt),(20,6*\rt)
					,
					(16.5,5*\rt),(17.5,5*\rt),(18.5,5*\rt),(19.5,5*\rt)
					,
					(15,4*\rt),(16,4*\rt),(17,4*\rt),(18,4*\rt),(19,4*\rt)
					,
					(13.5,3*\rt),(14.5,3*\rt),(15.5,3*\rt),(16.5,3*\rt),(17.5,3*\rt),(18.5,3*\rt)
					,
					(12,2*\rt),(13,2*\rt),(14,2*\rt),(15,2*\rt),(16,2*\rt),(17,2*\rt),(18,2*\rt)
					,
					(10.5,1*\rt),(11.5,1*\rt),(12.5,1*\rt),(13.5,1*\rt),(14.5,1*\rt),(15.5,1*\rt),(16.5,1*\rt),(17.5,1*\rt)
					,
					(9,0*\rt),(10,0*\rt),(11,0*\rt),(12,0*\rt),(13,0*\rt),(14,0*\rt),(15,0*\rt),(16,0*\rt),(17,0*\rt)
					, 
					(7.5,-\rt),(8.5,-\rt),(9.5,-\rt),(10.5,-\rt),(11.5,-\rt),(12.5,-\rt),(13.5,-\rt),(14.5,-\rt),(15.5,-\rt),(16.5,-\rt)
					,
					(6,-2*\rt),(7,-2*\rt),(8,-2*\rt),(9,-2*\rt),(10,-2*\rt),(11,-2*\rt),
					(12,-2*\rt),(13,-2*\rt),(14,-2*\rt),(15,-2*\rt),(16,-2*\rt)
				}
				{
					\lozr{\rl}{red!20!white}
				}
				\foreach \ll in
				{
					(1.5,9*\rt),(3.5,9*\rt),(5.5,9*\rt),(7.5,9*\rt),(9.5,9*\rt),(11.5,9*\rt),(13.5,9*\rt),
					(15.5,9*\rt),(17.5,9*\rt),(19.5,9*\rt),(21.5,9*\rt)
					,
					(2,8*\rt),(4,8*\rt),(6,8*\rt),(8,8*\rt),(10,8*\rt),(12,8*\rt),(14,8*\rt),(16,8*\rt),(18,8*\rt),
					(20,8*\rt)
					,
					(2.5,7*\rt),(4.5,7*\rt),(6.5,7*\rt),(8.5,7*\rt),(10.5,7*\rt),(12.5,7*\rt),(14.5,7*\rt),
					(16.5,7*\rt),(18.5,7*\rt)
					,
					(3,6*\rt),(5,6*\rt),(7,6*\rt),(9,6*\rt),(11,6*\rt),(13,6*\rt),(15,6*\rt),(17,6*\rt)
					,
					(3.5,5*\rt),(5.5,5*\rt),(7.5,5*\rt),(9.5,5*\rt),(11.5,5*\rt),(13.5,5*\rt),(15.5,5*\rt)
					,
					(4,4*\rt),(6,4*\rt),(8,4*\rt),(10,4*\rt),(12,4*\rt),(14,4*\rt)
					,
					(4.5,3*\rt),(6.5,3*\rt),(8.5,3*\rt),(10.5,3*\rt),(12.5,3*\rt)
					,
					(5,2*\rt),(7,2*\rt),(9,2*\rt),(11,2*\rt)
					,
					(5.5,1*\rt),(7.5,1*\rt),(9.5,1*\rt)
					,
					(6,0*\rt),(8,0*\rt)
					,
					(6.5,-\rt)
				}
				{
					\lozl{\ll}{yellow!20!white}
				}
				\foreach \vl in
				{
					(0,10*\rt),(2,10*\rt),(4,10*\rt),(6,10*\rt),(8,10*\rt),(10,10*\rt),
					(12,10*\rt),(14,10*\rt),(16,10*\rt),(18,10*\rt),(20,10*\rt),(22,10*\rt)
					,
					(0.5,9*\rt),(2.5,9*\rt),(4.5,9*\rt),(6.5,9*\rt),(8.5,9*\rt),(10.5,9*\rt),
					(12.5,9*\rt),(14.5,9*\rt),(16.5,9*\rt),(18.5,9*\rt),(20.5,9*\rt)
					,
					(1,8*\rt),(3,8*\rt),(5,8*\rt),(7,8*\rt),(9,8*\rt),(11,8*\rt),
					(13,8*\rt),(15,8*\rt),(17,8*\rt),(19,8*\rt)
					,
					(1.5,7*\rt),(3.5,7*\rt),(5.5,7*\rt),(7.5,7*\rt),(9.5,7*\rt),(11.5,7*\rt),
					(13.5,7*\rt),(15.5,7*\rt),(17.5,7*\rt)
					,
					(16,6*\rt),(2,6*\rt),(4,6*\rt),(6,6*\rt),(8,6*\rt),(10,6*\rt),
					(12,6*\rt),(14,6*\rt)
					,
					(14.5,5*\rt),(2.5,5*\rt),(4.5,5*\rt),(6.5,5*\rt),(8.5,5*\rt),(10.5,5*\rt),
					(12.5,5*\rt)
					,
					(13,4*\rt),(3,4*\rt),(5,4*\rt),(7,4*\rt),(9,4*\rt),(11,4*\rt)
					,
					(11.5,3*\rt),(3.5,3*\rt),(5.5,3*\rt),(7.5,3*\rt),(9.5,3*\rt)
					,
					(10,2*\rt),(4,2*\rt),(6,2*\rt),(8,2*\rt)
					,
					(4.5,1*\rt),(6.5,1*\rt),(8.5,1*\rt)
					,
					(5,0),(7,0)
					,
					(5.5,-\rt)
				}
				{
					\lozv{\vl}{green!20!white}
				}
			\end{scope}
		\end{tikzpicture}	
		\caption{%
			Lozenge tilings of minimal volume for the hexagon
			boundary conditions (left)
			and the regular density $\tfrac12$ top row (right).%
		}
	\label{fig:min_volume}
\end{figure}
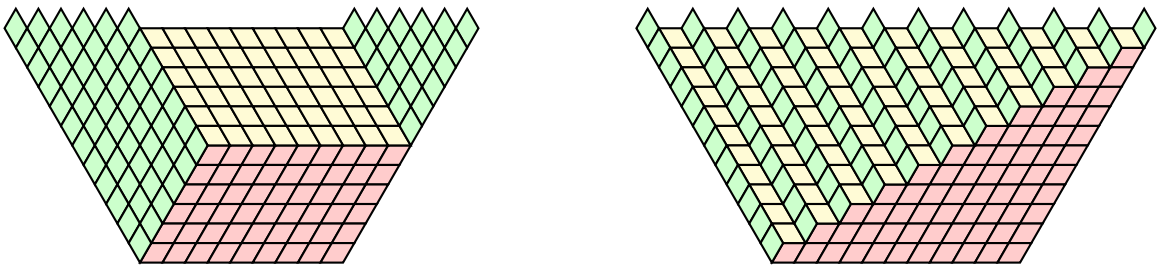

The situation in \Cref{fig:min_volume}, right,
is different in an essential way. 
Namely, for this top row the interlacing constraints allow much more configurations
with volume close to minimal, which suggests that the fluctuations 
of the vertical lozenges at finite distance from the bottom grow to infinity. 
This leaves open a possibility to have an ICLT
at some fluctuation scale $\ll \sqrt N$.
In terms of the weakly decreasing function $\mathsf{f}$ the
difference between the situations 
in the left and the right pictures in \Cref{fig:min_volume}
is in the value of $\mathsf{f}'(1)$ corresponding 
to the density of the vertical lozenges at the left end of the top row.
This leads to the following conjecture:

\begin{conjecture}
	Let Assumption \ref{ass:function} hold, and replace Assumption \ref{ass:q} 
	by $q=q(N)=e^{-\upgamma(N)/N}$, where $\upgamma(N)\to+\infty$, $\upgamma(N)\ll N$.
	If $\mathsf{f}'(1)<0$, then an analogue of \Cref{thm:main_result} holds
	under a suitable normalization determined by $\mathsf{u}(N)$ and
	$\upsigma^2(N)$ with 
	$
	\lim_{N\to\infty}\mathsf{u}(N)
	=
	\lim_{N\to\infty}\upsigma^2(N)=0
	$:
	\begin{equation*}
		\biggl\{ \frac{\lambda^r_j(N)-\mathsf{u}(N)N}{\upsigma(N)\sqrt{N}} \colon 1\le r\le K,\,1\le j\le r \biggr\}
		\to
		\left\{ \mathsf{L}^r_j\colon 1\le r\le K,\,1\le j\le r \right\}\sim
		\mathsf{GUE}_{K\times K}(1).
	\end{equation*}
\end{conjecture}

\subsection{Background and methods}
\label{sub:previous}

Let us discuss previous work on ICLT in random lozenge tilings 
and how our approach differs from the existing ones.

The ICLT has been proven for various ensembles of 
random
domino and lozenge tilings
starting from
\cite{Johansson2005arctic},
\cite{OkounkovReshetikhin2006RandomMatr},
\cite{johansson2006eigenvalues}, 
\cite{johansson2006eigenvaluesErratum}.
In particular, 
\cite{johansson2006eigenvalues} essentially establishes
an ICLT for uniformly random lozenge tilings of the hexagon
(see also \cite{Nordenstam2009Thesis} for more details and for a discussion of ICLT for domino tilings),
and 
\cite{OkounkovReshetikhin2006RandomMatr}
considers an ensemble of $q^{\mathsf{vol}}$-weighted random lozenge tilings of certain infinite regions.
The latter paper
also informally argues that 
the ICLT should hold for
Gibbs (in particular, uniform) ensembles
of random lozenge tilings at points
where two frozen regions meet (such points are sometimes called turning points),
see also \cite{borodinDuits2011GUE} for an ICLT for a Gibbs measure on tilings of an infinite region.
Here by a Gibbs property we mean that for each $K\ge1$
the distribution of levels $1,\ldots,K-1 $
of an interlacing array
(discrete or continuous) 
conditioned on fixed configuration at level $K$ is uniform 
over all configurations subject to interlacing.
In non-Gibbs situation the limit at turning points 
might not be the universal GUE corners process
but rather its certain splitting, cf. \cite{Mkrtchyan2014Periodic}.

The ICLT results cited in the previous paragraph rely on the presence of
explicit determinantal correlation kernels of the pre-limit models suitable for
asymptotic analysis.
Such kernels can be extracted from the formalism of Schur processes
\cite{okounkov2003correlation}, \cite{Okounkov2005}, \cite{borodin2005eynard}
for $q^{\mathsf{vol}}$-weighted or periodically weighted random tilings of
infinite regions; 
via connections with orthogonal polynomials \cite{johansson2006eigenvalues},
\cite{borodin-gr2009q} for $q^{\mathsf{vol}}$-weighted tilings of the hexagon;
or by an application of Eynard--Mehta type results
\cite{Petrov2012}, \cite{duse2015asymptotic}
for uniformly random tilings of general sawtooth domains.
The latter method also produces a determinantal correlation
kernel for the
$q^{\mathsf{vol}}$-weighted tilings of general sawtooth domains 
\cite{Petrov2012}, but it has a more complicated structure than in the uniform
case, which so far has presented an impediment to its asymptotic analysis.

\medskip

Recently other methods of proving ICLT (and other asymptotic results)
for uniformly random lozenge tilings of general sawtooth domains 
as in 
\Cref{ass:function}
were developed in 
\cite{GorinPanova2012_full} and
\cite{novak2015lozenge}.
The approach in the former paper
is based on Schur generating functions
and allows to also consider tilings with free boundary
\cite{Panova2014Free}
and to establish an ICLT for
alternating sign matrices
\cite{gorin2014alternating}.
The methods of \cite{novak2015lozenge} are
based on combinatorics of Hurwitz numbers.
These more recent methods are related to free probability and go beyond the
ICLT leading to limit shape results
\cite{GorinBufetov2013free},
\cite{Collins2016}
(both approaches)
and asymptotics of global fluctuations
\cite{bufetov2016fluctuations}
(the Schur generating functions approach).

\medskip

Our approach to the ICLT for $q^{\mathsf{vol}}$-weighted lozenge
tilings of general sawtooth domains (\Cref{thm:main_result})
is closer to methods relying on determinantal kernels.
However, while the determinantal structure
of the measure $q^{\mathsf{vol}}$ obtained in \cite{Petrov2012}
is quite complicated, the distribution of the vertical lozenges
at each $K$-th level of the interlacing array
can be written as a much simpler $K\times K$ determinant
\cite{Petrov2012GT} with matrix elements
expressed as single contour integrals. 
This is the starting point of our work, and the asymptotic
analysis of the single contour integrals yields our ICLT.

\subsection{Outline}
\label{sub:outline}

In \Cref{sec:prelimit_formula}
we recall the $K\times K$ determinantal formula from \cite{Petrov2012GT} for 
the distribution of the vertical lozenges at the $K$th level of the interlacing
array.
In \Cref{sec:asymptotic_mean} we perform a preliminary asymptotic computation
at the first ($K=1$) level
to determine the global location $\mathsf{u}$ \eqref{u_intro} 
of the vertical lozenges.
In \Cref{sec:critical_points} we 
analyze the critical points and the steepest descent contour of the 
integral entering the distribution of the vertical lozenges at the $K$th level.
In \Cref{sec:completing_proof} we compute the asymptotics
of the distribution of the vertical lozenges at any $K$th level
(which are dominated by the behavior in a neighborhood of a critical point),
and complete the proof of \Cref{thm:main_result}.

\subsection{Acknowledgments}

We are grateful to conversations with Vadim Gorin, Jonathan Novak, and
Greta Panova. 
SM was partially supported by the Simons Foundation Collaboration
Grant No. 422190.

\section{Projection of $P_{q,\nu}^N$ onto the $K$th row}
\label{sec:prelimit_formula}

The starting point of our analysis is a formula for the projection of our measure
$P^N_{q,\nu}$ onto any fixed row $K$, i.e., a formula for the 
distribution of the particles $\lambda^K_j$, $j=1,\ldots,K $,
under $P^{N}_{q,\nu}$ for any fixed $1\le K<N$.
Throughout the paper we use the notation
\begin{equation*}
(a;q)_{m}=(1-a)(1-aq)\ldots(1-aq^{m-1}),\qquad  m=1,2,\ldots
\end{equation*}
(with $(a;q)_{0}=1$), 
for the $q$-Pochhammer symbol.

\begin{theorem}[{\cite[Theorem 1.5]{Petrov2012GT}}]
	\label{thm:Pet14}
	For any $1\le K<N$, the distribution of the $K$th row under $P^{N}_{q,\nu}$ has the following
	determinantal form:
	\begin{equation}
		\label{formula_Pet14}
		\begin{split}
			&
			P^{N}_{q,\nu}\left( \lambda^K_j=\varkappa_j,\;j=1,\ldots,K  \right)
			\\&\hspace{40pt}=
			s_{\varkappa}(1,q,\ldots,q^{K-1})\,
			(-1)^{K(N-K)}
			q^{(N-K)|\varkappa|}
			q^{-K(N-K)(N+2)/2}
			\Det_{i,j=1}^{K}[ A_{i}(\varkappa_j-j)],
		\end{split}
	\end{equation}
	where $\varkappa_1\ge\ldots\ge\varkappa_K $ are fixed integers, 
	$s_{\varkappa}$ is a Schur polynomial 
	(%
		see \eqref{Schur_q_specialization} for an explicit product formula
		for 
		$s_{\varkappa}(1,q,\ldots,q^{K-1})$%
	),
	and $A_i(x)$, $i=1,\ldots,K $, $x\in\mathbb{Z}$, are the following functions:
	\begin{align}\label{qA_i}
		A_i(x)= A_i(x\mid K,N,\nu):=
		\frac{1-q^{N-K}}{2\pi\mathbf{i}}\oint_{C(x)}dz\,
		\frac{(zq^{1-x};q)_{N-K-1}}
		{\prod_{r=i}^{N-K+i}(z-q^{-r})}
		\prod_{r=1}^{N}\frac{z-q^{-r}}{z-q^{\nu_r-r}}.
	\end{align}
	Here the positively (counter-clockwise) oriented simple contour $C(x)$
	encircles points $q^{x},q^{x+1},\ldots,q^{\nu_1-1}$, and not the possible
	poles $q^{x-1},q^{x-2},\ldots,q^{\nu_N-N}$.
\end{theorem}

Let us rewrite the functions
$A_i(x)$ in a form more convenient for
asymptotic analysis:
\begin{proposition}
	\label{proposition:qA_i_best}
	We have
	\begin{equation}
		\label{qA_i_best}
		A_i(x)=
		\frac{q^{x}(-1)^{N-K-1}
		(1-q^{N-K})}{2\pi\mathbf{i}}
		\oint_{C'}dw
		\frac{\prod_{r=1}^{N-K-1}(q^r-w)}
		{\prod_{r=1}^N (q^{x}-wq^{\nu_r-r})}
		\prod_{r\in I_{i,N,K}}(q^{x}-wq^{-r}),
	\end{equation}
	where $I_{i,N,K}:=\left\{ 1,\ldots,i-1  \right\}\cup\left\{ N-K+i+1,\ldots,N  \right\}$
	and the contour $C'$ is a positively oriented circle crossing the real line
	to the left of $0$ and between $q$ and $1$.
\end{proposition}
\begin{proof}
	Change the variables in \eqref{qA_i} as $z=q^{x}w^{-1}$.
	Then $w^{-1}$ is integrated over a positively oriented circle 
	encircling $1,q,q^2,\ldots $ and not $q^{-1},q^{-2},\ldots $.
	That is, the integration contour for $w^{-1}$ encircles the segment $[0,1]$.
	Therefore, we can choose the contour $C'$ for $w$ as in the claim, and 
	note that there is an additional change of sign coming 
	from the orientation of the contour. 
	Thus,
	\begin{equation*}
		A_i(x)=
		\frac{q^x(1-q^{N-K})}{2\pi\mathbf{i}}\oint_{C'}\frac{dw}{w^2}
		\frac{(q/w;q)_{N-K-1}}
		{\prod_{r=i}^{N-K+i}(q^x/w-q^{-r})}
		\prod_{r=1}^{N}\frac{q^x/w-q^{-r}}{q^x/w-q^{\nu_r-r}}.
	\end{equation*}
	Rewriting the products under the integral, we get \eqref{qA_i_best}.
\end{proof}

\section{Global location $\mathsf{u}$ from first level asymptotics}
\label{sec:asymptotic_mean}

According to the desired claim of \Cref{thm:main_result},
under the measures $P^{N}_{q,\nu}$ and
for any fixed $K$ the random locations
of the vertical lozenges 
$\begin{tikzpicture}
	[scale=.17,very thick]
	\def\rt{0.866025}
	\lozv{(0,0)}{white}
\end{tikzpicture}$
on the $K$th row
should behave as\footnote{%
	Here and below for simplicity we omit the notation of the integer part,
	and write, e.g., 
	$\lambda^{K}_j=\mathsf{u}N+\mathsf{L}^K_j\sqrt N$
	instead of
	$\lambda^{K}_j=\lfloor\mathsf{u}N+\mathsf{L}^K_j\sqrt N\rfloor$.
	The presence of the integer parts does not affect the 
	asymptotics.
} 
$\lambda^{K}_j=\mathsf{u}N+\mathsf{L}^K_j\sqrt N$ as $N\to+\infty$
(note that $\mathsf{u}$ does not depend on $K$).
Here
$\mathsf{L}^K=(\mathsf{L}^{K}_1\ge \ldots\ge \mathsf{L}^K_K)$, $\mathsf{L}^K_j\in\mathbb{R}$,
are the rescaled locations which we will later identify with the 
eigenvalues of the $K\times K$ GUE random matrix (see \Cref{def:GUE}).
In this section we perform a preliminary computation for 
the simplest case $K=1$
which will help determine the correct global asymptotic location $\mathsf{u}$.

When $K=1$, the projection formula of 
\Cref{sec:prelimit_formula} becomes
\begin{equation}
	\label{integral_formula_for_K1}
	\begin{split}
		P^{N}_{q,\nu}
		\left( \lambda^1_1=x+1 \right)
		&=
		(-1)^{N-1}q^{(N-1)(x+1)}q^{-\frac{(N-1)(N+2)}2}
		A_1(x)
		\\&=
		-
		q^{-\frac{N(N-2x-1)}2}
		\frac{(1-q^{N-1})}{2\pi\mathbf{i}}
		\oint_{C'}dw
		\frac{\prod_{r=1}^{N-2}(q^r-w)}
		{\prod_{r=1}^N (q^{x}-wq^{\nu_r-r})}.
	\end{split}
\end{equation}
Since $x$ should depend on $N$ as $x=\mathsf{u}N+\xi\sqrt N$, 
we have from Assumption \ref{ass:q}:
\begin{equation}
	\label{integral_formula_for_K1_prefactor_behavior}
	-
	q^{-\frac{N(N-2x-1)}2}
	(1-q^{N-1})=
	-e^{-N\upgamma(\mathsf{u}-1/2)+O(\sqrt N)}
	(1-e^{-\upgamma}+O(1/N)).
\end{equation}
Therefore, the exponential in $N$ terms coming from the 
asymptotics of the contour integral should compensate 
$e^{-N\upgamma(u-1/2)}$.

Let us now consider the 
exponential in $N$ behavior
of 
the terms under the integral in \eqref{integral_formula_for_K1}.
The numerator behaves as
\begin{equation}
	\label{K1_int_expon1}
	\prod_{r=1}^{N-2}(q^r-w)
	=
	\exp
	\left\{ N\sum_{r=1}^{N-2}\frac{1}{N}\log
	\left( e^{-\upgamma\frac{r}{N}} -w\right) \right\}
	=
	\exp
	\left\{ N\int_0^1\log\left( e^{-\upgamma s}-w \right)ds+O(1) \right\},
\end{equation}
where in the second equality we used the approximation
of the integral by the corresponding Riemann sum (which has error $O(1/N)$).
Here and below we take $\log$ to be the standard branch having 
the cut along the negative real axis. 
Later in the course of our analysis it will become evident that
the integral over the part of the contour $C'$ 
around the branch cut (i.e., $e^{-\upgamma s}-w<0$)
is asymptotically negligible. 
This justifies our choice of the branch.

For the denominator in \eqref{integral_formula_for_K1} we have
\begin{align}
	\nonumber
	\prod_{r=1}^{N}(q^x-wq^{\nu_r-r})
	&=
	\exp
	\biggl\{ N\sum_{r=1}^{N}
	\log\left( 
		e^{-\upgamma\mathsf{u}-\frac{\upgamma\xi}{\sqrt N}} 
		-
		w e^{-\upgamma \frac{\nu_r}{N}+\upgamma\frac{r}{N}}
	\right)\biggr\}
	\\
	\label{K1_int_expon2}
	&=
	\exp
	\biggl\{ N\sum_{r=1}^{N}
	\log\left( 
		e^{-\upgamma\mathsf{u}} 
		-
		w e^{-\upgamma \mathsf{f}(\frac{r}{N})+\upgamma\frac{r}{N}}
	\right)
	+O(\sqrt N)\biggr\}
	\\&=
	\exp
	\biggl\{ 
		N\int_0^1\log\left( e^{-\upgamma\mathsf{u}}
		-we^{\upgamma(s-\mathsf{f}(s))}\right)ds
		+
		O(\sqrt N)
	\biggr\}.
	\nonumber
\end{align}
Having \eqref{K1_int_expon1}--\eqref{K1_int_expon2}, let us denote
\begin{equation}
	\label{S_action_function_definition}
	S(w)
	:=
	\int_0^1\log
	\left( e^{-\upgamma s}-w \right)ds
	-
	\int_0^1\log
	\left( e^{-\upgamma\mathsf{u}}
	-we^{\upgamma(s-\mathsf{f}(s))}\right)ds.
\end{equation}
Therefore, the integral in \eqref{integral_formula_for_K1}
takes the form
\begin{equation}\label{O_sqrt_N_part}
	\frac{1}{2\pi\mathbf{i}}
	\int_{C'} e^{NS(w)+O(\sqrt N)}dw.
\end{equation}
The asymptotic behavior of such integrals can be analyzed
using the steepest descent (also called Laplace)
method.
Namely, one finds a critical point $w_{\mathsf{cr}}$ 
(in our case this indeed will be a single critical point)
and a deformation
of the integration contour $C'$
so that the deformed contour $\mathfrak{C}$ passes
through $w_{\mathsf{cr}}$ and is a steepest descent contour for $S(w)$
in the sense that 
$\Re\left( S(w)-S(w_{\mathsf{cr}}) \right)< 0$, 
$w\in \mathfrak{C}\setminus\left\{ w_{\mathsf{cr}} \right\}$.
Then 
\begin{equation*}
	\frac{1}{2\pi\mathbf{i}}
	\int_{C'} e^{NS(w)+O(\sqrt N)}dw
	=
	\frac{e^{NS(w_{\mathsf{cr}})}}{2\pi\mathbf{i}}
	\int_{C'} e^{N\left( S(w)-S(w_{\mathsf{cr}}) \right)+O(\sqrt N)}dw,
\end{equation*}
so the exponential in $N$ terms in the integral have the form
$e^{NS(w_{\mathsf{cr}})}$.
Comparing this to \eqref{integral_formula_for_K1_prefactor_behavior} we see that
one should find $\mathsf{u}$ and $w_{\mathsf{cr}}$ such that
\begin{equation}
	\label{u_wcr_equations}
	S'(w_{\mathsf{cr}})=0,\qquad 
	S(w_{\mathsf{cr}})=\upgamma\left( \mathsf{u}-\frac{1}{2} \right).
\end{equation}
The next lemma provides one such pair $(\mathsf{u},w_{\mathsf{cr}})$ which 
will turn out to be the right one for asymptotics.
\begin{lemma}
	\label{lemma:S_at_0}
	We have $S(0)=\upgamma(\mathsf{u}-1/2)$ and 
	\begin{equation*}
		S'(0)=-\int_0^1 e^{\upgamma s}ds+e^{\upgamma \mathsf{u}}\int_0^1 e^{\upgamma(s-\mathsf{f}(s))}ds.
	\end{equation*}
\end{lemma}
\begin{proof}
	A straightforward computation.
\end{proof}
Thus, the pair
\begin{equation}
	\label{u_correct_value}
	\mathsf{u}
	=
	\frac{1}{\upgamma}
	\log
	\left( 
		\frac{\int_0^1 e^{\upgamma s}ds}
		{\int_0^1 e^{\upgamma(s-\mathsf{f}(s))}ds}
	\right),
	\qquad 
	w_{\mathsf{cr}}=0
\end{equation}
solves \eqref{u_wcr_equations}. 
Throughout the rest of the paper we will fix the value of $\mathsf{u}$
as in \eqref{u_correct_value} and will call it the global location.

\begin{lemma}
	\label{lemma:u_obvious_inequalities}
	We have $0\le\mathsf{u}\le\mathsf{f}(0)$.
\end{lemma}
\begin{proof}
	First, since $0=\mathsf{f}(1)\le \mathsf{f}(s)\le\mathsf{f}(0)$ for all $s\in[0,1]$,
	the bottom integral under the logarithm in \eqref{u_correct_value} is no greater than the top one,
	so $\mathsf{u}\ge0$. 
	The inequality $\mathsf{u}\le \mathsf{f}(0)$ follows from
	$\int_0^1 e^{\upgamma (s-\mathsf{f}(s))}ds
	\ge e^{-\upgamma\mathsf{f}(0)}\int_0^1 e^{\upgamma s}ds$, 
	which completes the proof.
\end{proof}
\Cref{lemma:u_obvious_inequalities} 
confirms that the global location $\mathsf{u}$ lies
within the range prescribed by the interlacing constraints
$0=N\mathsf{f}(1)\approx\nu_N\le \lambda^{K}_j\le \nu_1\approx N\mathsf{f}(0)$.

\section{Critical points and contour deformation}
\label{sec:critical_points}

In this section we study in detail the critical points of the function $S(w)$
as well as the contour plot of $\Re S(w)$.
One of the key ingredients is an approximation of the
weakly decreasing function $\mathsf{f}(x)$, $x\in[0,1]$,
of Assumption \ref{ass:function} by piecewise constant 
weakly decreasing functions $\mathsf{f}_m(x)$ 
(we also require that each $\mathsf{f}_m$ has a finite range).
The approximating functions $\mathsf{f}_m(x)$ do not depend on $N$,
and the whole approximation is done independently of the main limit $N\to+\infty$.
Note that the case of $\mathsf{f}(x)$ piecewise constant corresponds to
considering random tilings of polygons with fixed number of proportionally growing sides.

\subsection{Critical points}
\label{sub:critical_piecewise}

Let us assume that $\mathsf{f}(x)=\mathsf{f}_m(x)$ 
is piecewise constant with $m\ge 2$ values, and parametrize
it by two sequences
\begin{equation}
	\label{piecewise_f_1}
	0=s_0<s_1<\ldots <s_m=1,\qquad \alpha_1>\alpha_2>\ldots >\alpha_m=0
\end{equation}
such that (here and below $\mathbf{1}_{A}$ is the indicator of a set $A$)
\begin{equation}
	\label{piecewise_f_2}
	\mathsf{f}_m(x)=\sum_{j=1}^m \alpha_j\mathbf{1}_{s_{j-1}<x\le s_j}.
\end{equation}
Recall that we always assume that $\mathsf{f}(1)=0$.
Note that $\left\{ s_j \right\}$ and $\left\{ \alpha_j \right\}$ will depend on $m$, 
but we suppress this dependence in the notation.

In this piecewise constant case the function $S(w)=S_m(w)$ \eqref{S_action_function_definition}
becomes
\begin{equation}
	\label{S_action_piecewise}
	S_m(w)=
	\int_0^1\log
	\left( e^{-\upgamma s}-w \right)ds
	-
	\sum_{j=1}^{m}
	\int_{s_{j-1}}^{s_j}\log
	\left( e^{-\upgamma\mathsf{u}}
	-we^{\upgamma(s-\alpha_j)}\right)ds,
\end{equation}
and each of the integrals can be expressed in terms of dilogarithms. 
The $w$-derivative of $S_m(w)$ has a simpler form:
\begin{lemma}
	\label{lemma:critical_points_equation}
	We have
	\begin{equation}
		\label{S_m_prime}
		S_m'(w)=
		\frac{\log \left(1-we^{\upgamma } \right)-\log (1-w)}{\upgamma  w}
		+
		\sum_{j=1}^{m}
		\frac{\log \left(1-w e^{\upgamma  
		(\mathsf{u}-\alpha_j +s_{j-1})}
		\right)-
		\log \left(1-w e^{\upgamma  
		(\mathsf{u}-\alpha_j+s_j)}
		\right)}
		{\upgamma  w}.
	\end{equation}
	Moreover, the critical point equation $S_m'(w)=0$ implies
	an algebraic equation
	\begin{equation}
		\label{S_m_algebraic}
		e^{\upgamma wS_m'(w)}=
		\frac{1-we^{\upgamma}}{1-w}
		\prod_{j=1}^{m}
		\frac{1-we^{\upgamma(\mathsf{u}-\alpha_j+s_{j-1})}}
		{1-we^{\upgamma(\mathsf{u}-\alpha_j+s_j)}}
		=
		1.
	\end{equation}
\end{lemma}
\begin{proof}
	A straightforward computation.
\end{proof}
The expression for $S_m'(w)$ given in \eqref{S_m_prime} is valid with standard
branches of the logarithms for $w$ sufficiently close to $0$.
Instead of looking at \eqref{S_m_prime} directly we will focus on the algebraic
equation \eqref{S_m_algebraic} implied by $S_m'(w)=0$
in part because \eqref{S_m_algebraic} does not involve choosing the branches.

Clearly, $w=0$ satisfies \eqref{S_m_algebraic} regardless 
of the value of $\mathsf{u}$ (but depending on $\mathsf{u}$ the solution
$w=0$ might not lead to a critical point of $S_m(w)$).
Let us examine other possible critical points:
\begin{proposition}
	\label{proposition:other_critical_points_piecewise}
	All solutions to \eqref{S_m_algebraic} are real.
	Moreover, besides $w=0$, all solutions 
	to \eqref{S_m_algebraic}
	belong
	to the segment
	$\left[ e^{-\upgamma(\mathsf{u}+1)},
	e^{-\upgamma(\mathsf{u}-\alpha_1)} \right]$.
\end{proposition}
Note that $e^{-\upgamma(\mathsf{u}+1)}<1$, and 
$
e^{-\upgamma(\mathsf{u}-\alpha_1)}
=
e^{-\upgamma(\mathsf{u}-\mathsf{f}_m(0))}\ge1
$, 
cf.
\Cref{lemma:u_obvious_inequalities}.
\begin{proof}[Proof of \Cref{proposition:other_critical_points_piecewise}]
	Let $N_m(w)$ be the numerator and $D_m(w)$ the denominator of \eqref{S_m_algebraic}.
	Both are polynomials in $w$ of degree $m+1$.
	We will count the number of real solutions to the equation $N_m(w)=D_m(w)$.

	First, note that the ratio of the leading coefficients in $N_m(w)$ and $D_m(w)$ is
	\begin{equation*}
		\frac{e^{\upgamma}\prod_{j=1}^m e^{\upgamma(\mathsf{u}-\alpha_j+s_{j-1})}}
		{\prod_{j=1}^{m}e^{\upgamma(\mathsf{u}-\alpha_j+s_j)}}
		=e^{\upgamma(1+s_0-s_m)}=1
	\end{equation*}
	by \eqref{piecewise_f_1}.
	Thus, $N_m(w)-D_m(w)$ is a polynomial of degree at most $m$, so the equation
	$N_m(w)=D_m(w)$ has at most $m$ solutions. 

	Observe that in our piecewise constant case the 
	expression \eqref{u_correct_value} for $\mathsf{u}$
	simplifies to 
	\begin{equation}
		\label{u_correct_value_piecewise}
		\mathsf{u}
		=
		\frac{1}{\upgamma}
		\log
		\left( 
			\frac{e^{\upgamma}-1}{\sum_{j=1}^{m}
			\left( 
				e^{\upgamma(s_j-\alpha_j)}
				-
				e^{\upgamma(s_{j-1}-\alpha_j)}
			\right)}
		\right).
	\end{equation}
	This readily implies that $N_m(w)$ and $D_m(w)$ have
	the same linear in $w$ terms. 
	Since, in addition, the constant terms are both $1$, we
	see that $w=0$ is a double solution to $N_m(w)=D_m(w)$.

	Consider separately the roots of $N_m(w)$ and $D_m(w)$,
	and denote $n_j:=e^{-\upgamma(\mathsf{u}-\alpha_j+s_{j-1})}$
	and $d_j:=e^{-\upgamma(\mathsf{u}-\alpha_j+s_j)}$, $j=1,\ldots,m $.
	Note that $N_m(w)$ in addition has the root $w=e^{-\upgamma}$,
	and $D_m(w)$ in addition has the root $w=1$.
	From \eqref{piecewise_f_1} we see that the other roots interlace as
	\begin{equation*}
		d_m<n_{m}<d_{m-1}<n_{m-1}<\ldots<n_2<d_1<n_1.
	\end{equation*}
	Moreover, by \eqref{piecewise_f_1} and \Cref{lemma:u_obvious_inequalities} 
	we have
	\begin{equation*}
		d_m
		=
		e^{-\upgamma(\mathsf{u}+1)}
		\le 
		e^{-\upgamma}
		<
		1
		\le
		e^{-\upgamma(\mathsf{u}-\alpha_1)}
		=n_1.
	\end{equation*}
	Thus, all roots of $N_m(w)$ and $D_m(w)$
	(and not just the $n_j$'s and the $d_j$'s)
	``almost interlace''.
	Namely, 
	there is exactly one consecutive
	pair of roots of $N_m(w)$, one of them is $e^{-\upgamma}$, 
	surrounded by a pair of roots of $D_m(w)$, and 
	similarly there is only one consecutive pair of 
	roots of $D_m(w)$, one of them is $1$, 
	surrounded by a pair of roots of $N_m(w)$.
	See \Cref{fig:roots}.

	\begin{figure}[htpb]
		\begin{tikzpicture}
			[scale=.5,very thick]
			\draw[->,opacity=.4] (0,0)--++(30,0);
			\foreach \solid in
			{2,4,6,9,13,15,17,18,20,22.5,26}
			{
				\draw[fill] (\solid,0) circle(6pt);
			}
			\foreach \hollow in
			{3,5,7.3,10,11.5,14,16,19,21,24,27}
			{
				\draw (\hollow,0) circle(6pt);
			}
			\draw (1,1) 
				to[out=-25,in=100] (2,0) 
				to[out=-70,in=-95] (4,0)
				to[out=85,in=95] (6,0)
				to[out=-85,in=-95] (9,0)
				to[out=85,in=95] (13,0)
				to[out=-85,in=-95] (15,0)
				to[out=85,in=95] (17,0) to [out=-85,in=-180] (17.5,-1)
				to[out=0,in=-95] (18,0)
				to[out=85,in=95] (20,0)
				to[out=-85,in=-95] (22.5,0)
				to[out=85,in=105] (26,0)
				to[out=-75,in=140] (28,-2);
			\draw[densely dotted] (1,2)
				to[out=-35,in=110] (3,0) 
				to [out=-70,in=-180] (4,-1.5) to[out=0,in=-93] (5,0)
				to [out=70,in=-180] (6,1) to[out=0,in=93] (7.3,0)
				to [out=-70,in=-180] (8.5,-1.5) to[out=0,in=-93] (10,0)
				to [out=70,in=-180] (10.7,.8) to[out=0,in=93] (11.5,0)
				to [out=-70,in=-180] (12.5,-.8) to[out=0,in=-93] (14,0)
				to [out=70,in=-180] (15,2) to[out=0,in=93] (16,0)
				to [out=-70,in=-180] (17.5,-1.5) to[out=0,in=-93] (19,0)
				to [out=70,in=-180] (20,1.5) to[out=0,in=93] (21,0)
				to [out=-70,in=-180] (22,-1.5) to[out=0,in=-93] (24,0)
				to [out=70,in=-180] (25.5,1.5) to[out=0,in=93] (27,0)
				to [out=-70,in=140] (28,-1);
		\end{tikzpicture}
		\caption{%
			Positioning of the roots of $N_m(w)$ (hollow)
			and $D_m(w)$ (solid) for $m=10$, as well as schematic plots of 
			both polynomials.
			One readily sees that these plots must intersect 
			at least $m-2=8$ times.%
		}
		\label{fig:roots}
	\end{figure}
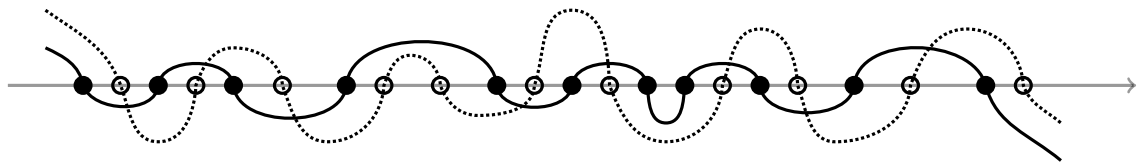
	
	This almost interlacing implies that there are 
	at least $m-2$ real solutions to $N_m(w)=D_m(w)$ lying from
	$e^{-\upgamma(\mathsf{u}+1)}$ to 
	$e^{-\upgamma(\mathsf{u}-\alpha_1)}$.
	Adding to this the double solution $w=0$
	we see that the claim holds.
\end{proof}

\begin{proposition}
	\label{proposition:other_critical_points_general}
	For an arbitrary weakly decreasing function $\mathsf{f}(x)$ on $[0,1]$, 
	the corresponding function $S(w)$ \eqref{S_action_function_definition}
	has no nonreal critical points. 
	Moreover, the only real
	critical point of $S(w)$
	outside the segment
	$\left[ e^{-\upgamma(\mathsf{u}+1)},
	e^{-\upgamma(\mathsf{u}-\mathsf{f}(0))} \right]$
	is $w_{\mathsf{cr}}=0$.
\end{proposition}
\begin{proof}
	If $\mathsf{f}(x)=\mathsf{f}_m(x)$ is piecewise constant, this follows from 
	\Cref{proposition:other_critical_points_piecewise} and the fact that the 
	equation $S_m'(w)=0$ implies \eqref{S_m_algebraic}.

	If $\mathsf{f}(x)$ is arbitrary, let us approximate it by a sequence
	of piecewise constant decreasing functions $\mathsf{f}_m(x)$, $m\to+\infty$,
	such that the corresponding derivatives $S_m'(w)$ converge 
	to $S'(w)$ uniformly in $w$ belonging to compact subsets of $\mathbb{C}\setminus\mathbb{R}$.
	If there is a nonreal critical point of $S(w)$, then by Rouche's theorem
	one of the functions $S_m(w)$ also would have a nonreal critical point, which is impossible
	by \Cref{proposition:other_critical_points_piecewise}.
	
	To control locations of real critical points note that $S'(w)$ is well
	defined (by differentiating under the integrals in
	\eqref{S_action_function_definition}) for real $w$ outside $[
	e^{-\upgamma(\mathsf{u}+1)}, e^{-\upgamma(\mathsf{u}-\mathsf{f}(0))} ]$, and
	so can be uniformly approximated there by the $S_m'(w)$'s. 
	This completes the proof.
\end{proof}

\subsection{Steepest descent contour}
\label{sub:contour}

After describing the critical points of the function $S(w)$
let us focus on the steepest descent contour 
$\mathfrak{C}:=\left\{ w\in\mathbb{C}\colon \Im S(w)=0 \right\}$
passing through the critical point $w_{\mathsf{cr}}=0$
and along which $\Re S(w)$ attains its only maximum at $w_{\mathsf{cr}}$
(recall that $S(0)=\upgamma(\mathsf{u}-1/2)$ is real).
The goal of this subsection is to justify that the contour $\mathfrak{C}$
and the region where $\Re S(w)<S(0)$ look as in 
\Cref{fig:steep_contour}.

\begin{figure}[htbp]
	\includegraphics[width=.4\textwidth]{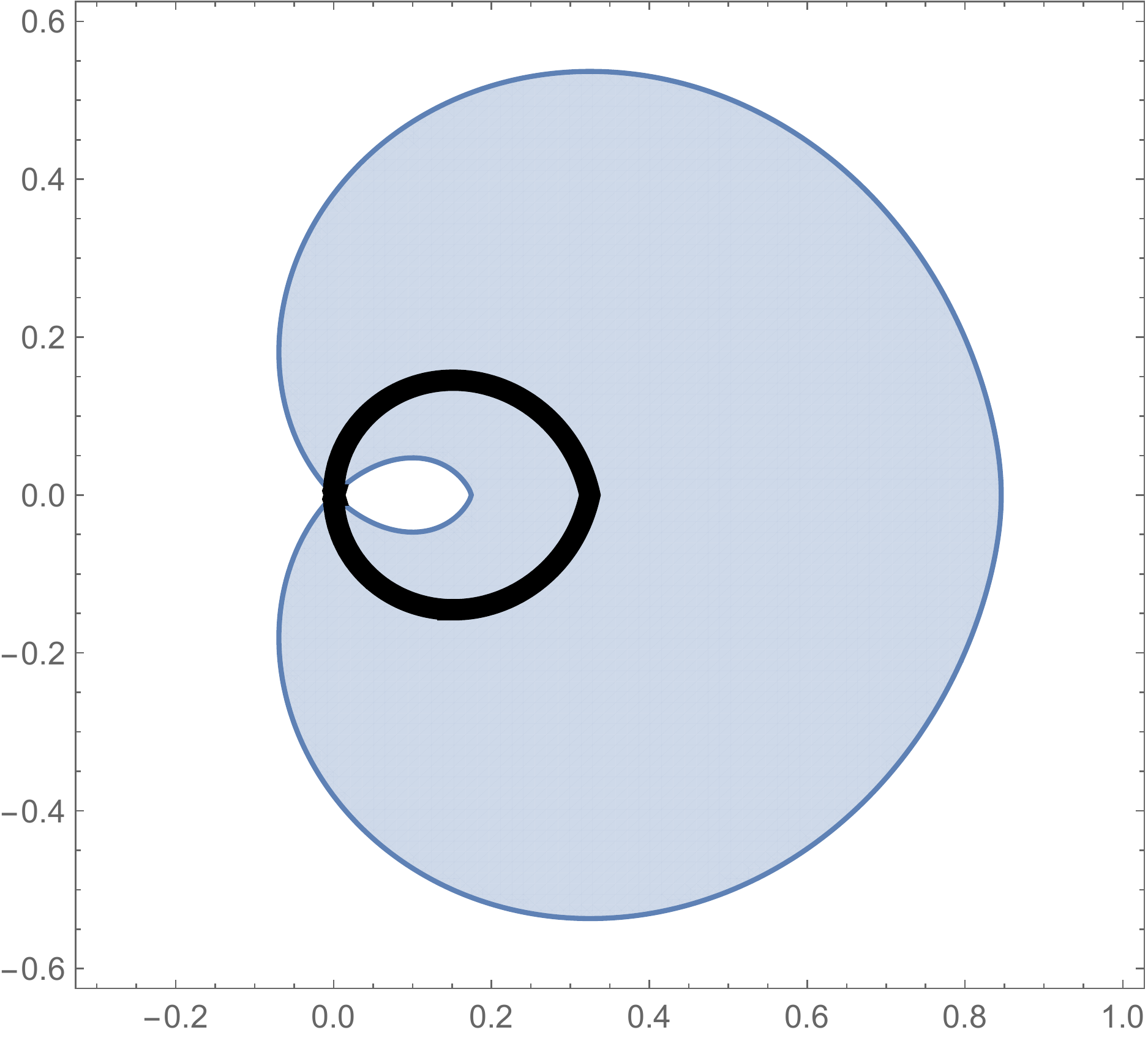}
	\caption{%
		The steepest descent contour $\mathfrak{C}=\{w\colon \Im S(w)=0\}$
		passing through the origin. 
		The shaded zone 
		is where $\Re S(w)<\Re S(0)$.
		The parameters in the picture
		are $m=2$, $\alpha=\{1,0\}$, $s=\{0,\frac34,1\}$, and $\upgamma=\frac{3}{2}$
		(so $e^{-\upgamma}\approx.223$).%
	}
	\label{fig:steep_contour}
\end{figure}

Our first observation is that the direction 
at which the steepest descent contour $\mathfrak{C}$
passes through zero is vertical:
\begin{lemma}
	\label{lemma:S_double_prime}
	We have $S''(0)>0$.
\end{lemma}
\begin{proof}
	First, we deal with the case when 
	$\mathsf{f}(x)=\mathsf{f}_m(x)$ is piecewise constant as in \eqref{piecewise_f_2}.
	Differentiating \eqref{S_m_prime} in $w$ and setting $w=0$ we obtain
	\begin{equation*}
		2\upgamma S_m''(0)
		=
		1-e^{2\upgamma}+
		\sum_{j=1}^{m}
		\left( 
			e^{2\upgamma(\mathsf{u}-\alpha_j+s_{j})}
			- 
			e^{2\upgamma(\mathsf{u}-\alpha_j+s_{j-1})}
		\right).
	\end{equation*}
	Plugging in $\mathsf{u}$ given by \eqref{u_correct_value_piecewise}
	we have
	\begin{equation*}
		2\upgamma S_m''(0)
		=
		\left( 1-e^{\upgamma} \right)
		\left[ 
			1+e^{\upgamma}
			-
			\left( e^{\upgamma}-1 \right)
			\frac{\sum_{j=1}^{m}
			e^{-2\upgamma\alpha_j}\left( e^{2\upgamma s_j}-e^{2\upgamma s_{j-1}} \right)}
			{
				\bigl( 
					\sum_{j=1}^{m}
					e^{-\upgamma\alpha_j}
					\left( e^{\upgamma s_j}-e^{\upgamma s_{j-1}} \right)
				\bigr)^2
			}
		\right].
	\end{equation*}
	We will prove that $S_m''(0)>0$ by induction on $m$. When $m=1$, using
	\eqref{piecewise_f_1} 
	we have $S_1''(0)=0$. 
	For general $m$, 
	$S_m''(0)$ 
	depends on $\alpha_1$
	in the following way:
	\begin{equation*}
		\frac{2\upgamma S_m''(0)}{1-e^{\upgamma}}
		=
		1+e^{\upgamma}
		-
		\left( e^{\upgamma}-1 \right)
		\frac{C_1 A^2+D_1}{(C_2A+D_2)^2},
		\qquad 
		A:=e^{-\upgamma \alpha_1},
	\end{equation*}
	where $C_{1,2}$ and $D_{1,2}$ do not depend on $\alpha_1$.
	Differentiating this with respect to $A$, we obtain
	\begin{equation}
		\label{S_double_prime_proof}
		\frac{2\upgamma}{1-e^{\upgamma}}
		\frac{\partial S_m''(0)}{\partial A}
		=
		2(e^{\upgamma}-1)
		\frac{C_2D_1-C_1D_2A}{(C_2A+D_2)^{3}}.
	\end{equation}
	The denominator is equal to the cube of
	\begin{equation*}
		C_2A+D_2
		=
		\sum_{j=1}^{m}
		e^{-\upgamma \alpha_j}
		\left( e^{\upgamma s_j}-e^{\upgamma s_{j-1}} \right),
	\end{equation*}
	which is positive by 
	\eqref{piecewise_f_1}.
	Therefore, the sign of the left-hand side of \eqref{S_double_prime_proof}
	is the same as the sign of
	\begin{align*}
		&
		C_2D_1
		-
		C_1D_2A
		\\&
		\hspace{20pt}=
		\left( e^{\upgamma s_1}-1 \right)
		\sum_{j=2}^{m}
		e^{-2\upgamma\alpha_j}\left( e^{2\upgamma s_j}-e^{2\upgamma s_{j-1}} \right)
		-
		\left( e^{2\upgamma s_1}-1 \right)e^{-\upgamma\alpha_1}
		\sum_{j=2}^{m}
		e^{-\upgamma\alpha_j}\left( e^{\upgamma s_j}-e^{\upgamma s_{j-1}} \right)
		\\&
		\hspace{20pt}=
		\left( e^{\upgamma s_1}-1 \right)
			\sum_{j=2}^{m}
			e^{-\upgamma\alpha_j}\left( e^{\upgamma s_j}-e^{\upgamma s_{j-1}} \right)
			\Bigl[
				e^{-\upgamma\alpha_j}
				\left( e^{\upgamma s_j}+e^{\upgamma s_{j-1}} \right)
				-
				e^{-\upgamma\alpha_1}
				\left( e^{\upgamma s_1}+1 \right)
			\Bigr].
	\end{align*}
	From \eqref{piecewise_f_1} we see that each summand above is positive.
	Therefore, $S_m''(0)$ decreases in $A=e^{-\upgamma\alpha_1}$, hence it increases in $\alpha_1$.
	Since our goal is to show that $S_m''(0)>0$ and we have $\alpha_1>\alpha_2$, it's
	enough to prove $S_m''(0)\geq 0$ when $\alpha_1=\alpha_2$. However, the latter
	corresponds to reducing $m$ by one. 
	Hence, by induction, we are done with the case of a piecewise constant
	$\mathsf{f}$.

	\medskip
	
	For an arbitrary decreasing $\mathsf{f}$, let $\mathsf{f}_m$ be a sequence of
	piecewise constant weakly decreasing functions converging to $\mathsf{f}$,
	such that $S_{m}(w)\rightarrow S(w)$ 
	uniformly on compact
	sets away from $[e^{-\upgamma
	(\mathsf{u}+1)},e^{-\upgamma(\mathsf{u}-\mathsf{f}(0))}]$. 

	From $S_{m}''(0)>0$ we have $S''(0)\geq 0$. 
	Suppose $S''(0)=0$,
	i.e., $0$ is a double critical point.
	Since $0$ is a simple critical point of $S_{m}$,
	by Hurwitz's theorem, if $m$ is large enough then
	$S_{m}$ must have a nonzero critical point in a 
	$\frac 12 e^{-\upgamma(\mathsf{f}(0)+1)}$ neighborhood of $0$.

	However, 
	\Cref{proposition:other_critical_points_piecewise} implies that all
	nonzero critical points of $S_{m}$ are larger than
	$e^{-\upgamma(\mathsf{u}_m+1)}$,
	where $\mathsf{u}_m$ is given by \eqref{u_correct_value_piecewise}
	for a piecewise constant $\mathsf{f}_m$.
	We have $\mathsf{u}_m\rightarrow \mathsf{u}$ because $\mathsf{f}_m\to\mathsf{f}$. 
	\Cref{lemma:u_obvious_inequalities}
	implies that $\mathsf{u}<\mathsf{f}(0)$. Hence
	\begin{equation*}
		e^{-\upgamma(\mathsf{u}+1)}>e^{-\upgamma(\mathsf{f}(0)+1)}.
	\end{equation*}
	It follows that if $m$ is large, all nonzero critical points of
	$S_{m}$ are larger than $\frac
	12e^{-\upgamma(\mathsf{f}(0)+1)}$, 
	so $S_{m}$ cannot have
	nonzero critical points in the $\frac 12e^{-\upgamma(\mathsf{f}(0)+1)}$
	neighborhood of $0$ as $m\rightarrow\infty$. This is a contradiction which
	completes the proof of the lemma.
\end{proof}

The steepest descent contour $\mathfrak{C}$ passing through the origin
in the vertical direction could in principle escape to infinity.
Let us show that this is not the case:

\begin{lemma}
	\label{lemma:behavior_at_infinity}
	We have $\lim_{|w|\to\infty}\Re S(w)>S(0)$.
\end{lemma}
\begin{proof}
	From \eqref{S_action_function_definition} we have
	\begin{equation*}
		\lim_{|w|\to\infty}\Re S(w)
		=
		\lim_{|w|\to\infty}
		\int_0^1
		\log\left|
		\frac{w-e^{-\upgamma s}}{we^{\upgamma(s-\mathsf{f}(s))}-e^{-\upgamma \mathsf{u}}}
		\right|ds
		=
		\upgamma\int_0^1\left( \mathsf{f}(s)-s \right)ds.
	\end{equation*}
	Recall that $S(0)=\upgamma(\mathsf{u}-\frac{1}{2})$, so we need to show that
	$\bar{\mathsf{f}}:=\int_0^1\mathsf{f}(s)ds>\mathsf{u}$. 
	From the definition of $\mathsf{u}$ \eqref{u_correct_value},
	the desired inequality is equivalent to
	\begin{equation*}
		e^{-\upgamma \,\bar{\mathsf{f}}}
		<
		\int_0^1 e^{-\upgamma \mathsf{f}(s)}d\mu(s),
	\end{equation*}
	where $d\mu(s)=\dfrac{\upgamma e^{\upgamma s}}{e^{\upgamma}-1}ds$ is a probability density
	on $[0,1]$. Thus, by Jensen's inequality for the strictly convex function $s\mapsto e^{-\upgamma s}$ we have
	\begin{equation*}
		\int_0^1 e^{-\upgamma \mathsf{f}(s)}d\mu(s)
		>
		\exp
		\left\{ -\upgamma \int_0^1 \mathsf{f}(s)d\mu(s) \right\}.
	\end{equation*}
	It remains to show that 
	\begin{equation}
		\label{behavior_at_infinity_proof}
		\int_0^1\mathsf{f}(s)d\mu(s)<\bar{\mathsf{f}}=\int_0^1\mathsf{f}(s)ds.
	\end{equation}
	Observe that the density of $d\mu(s)$ is increasing and $\mathsf{f}(s)$ is decreasing.
	Thus, in the left-hand side of \eqref{behavior_at_infinity_proof} 
	smaller values of $\mathsf{f}(s)$ are integrated with higher probabilistic weight
	coming from $d\mu(s)$ than in the right-hand side in which all values 
	of $\mathsf{f}(s)$ are integrated with uniform weight. 
	This implies \eqref{behavior_at_infinity_proof} and completes the proof.
\end{proof}

Therefore, the steepest descent contour $\mathfrak{C}$ emanating from $0$
must intersect the real line again.
The next lemma specifies where this point of intersection is:
\begin{lemma}
	\label{lemma:intersect_real_line}
	The contour $\mathfrak{C}=\left\{ w\colon \Im S(w)=0 \right\}$ passing through $0$
	intersects the real line again at a point between $e^{-\upgamma}$ and $1$.
\end{lemma}
\begin{proof}
	We will prove the lemma by looking at the behavior of 
	$\Im S_m(x+\mathbf{i}\epsilon)$ for $x\in \mathbb{R}$ and small $\epsilon>0$.
	First, observe that 
	\begin{equation*}
		\Im \log(x+\mathbf{i}\epsilon)=\arg(x+\mathbf{i}\epsilon)
		\sim \pi\mathbf{1}_{x<0}.
	\end{equation*}
	For piecewise constant $\mathsf{f}=\mathsf{f}_m$ \eqref{piecewise_f_2} 
	this implies that
	\begin{equation*}
		\frac{1}{\pi}\Im S_m(x+\mathbf{i}\epsilon)
		\sim
		\int_0^1
		\mathbf{1}_{x>e^{-\upgamma s}}
		ds
		-
		\sum_{j=1}^{m}
		\int_{s_{j-1}}^{s_j}
		\mathbf{1}_{x>e^{-\upgamma(s+\mathsf{u}-\alpha_j)}}
		ds.
	\end{equation*}
	Note that the right-hand side is
	continuous in $x$.
	Moreover, it is piecewise linear in $\log x$.
	From \eqref{piecewise_f_1} it follows that
	\begin{equation*}
		e^{-\upgamma(s_m+\mathsf{u}-\alpha_m)}
		<
		e^{-\upgamma(s_{m-1}+\mathsf{u}-\alpha_m)}
		<
		e^{-\upgamma(s_{m-1}+\mathsf{u}-\alpha_{m-1})}
		<\ldots< 
		e^{-\upgamma(s_{1}+\mathsf{u}-\alpha_1)}
		<
		e^{-\upgamma(s_{0}+\mathsf{u}-\alpha_1)},
	\end{equation*}
	and
	\Cref{lemma:u_obvious_inequalities} 
	also implies that 
	\begin{equation*}
		e^{-\upgamma(\mathsf{u}+1)}
		=
		e^{-\upgamma(s_m+\mathsf{u}-\alpha_m)}
		<
		e^{-\upgamma}
		<
		1
		<
		e^{-\upgamma(s_0+\mathsf{u}-\alpha_1)}
		=
		e^{-\upgamma(\mathsf{u}-\mathsf{f}_m(0))}
		.
	\end{equation*}
	This implies that 
	$\frac{1}{\pi}\Im S_m(x+\mathbf{i}\epsilon)$ is
	(see \Cref{fig:plot_Im} for an illustration):
	\begin{enumerate}[\quad$\bullet$]
		\item zero when $x<e^{-\upgamma(s_m+\mathsf{u}-\alpha_m)}$;
		\item weakly decreasing when $e^{-\upgamma(s_m+\mathsf{u}-\alpha_m)}
			<x<e^{-\upgamma}$;
		\item weakly increasing when $e^{-\upgamma}<x<1$;
		\item weakly decreasing when $1<x<e^{-\upgamma(s_0+\mathsf{u}-\alpha_1)}$;
		\item and zero when $x>e^{-\upgamma(s_0+\mathsf{u}-\alpha_1)}$.
	\end{enumerate}
	We see that there must be at least one 
	point in $(e^{-\upgamma},1)$ for which
	$\Im S_m(x+\mathbf{i}\epsilon)=0$. 
	
	There are two cases. 
	First, assume that $\mathfrak{C}$ intersects $\mathbb{R}$ at a 
	critical point of $S_m$. 
	Then the limit of $\Im S_m(x+\mathbf{i}\epsilon)$ as $\epsilon\searrow0$
	is zero in a neighborhood (within $\mathbb{R}$) of this point of intersection. 
	From \Cref{proposition:other_critical_points_piecewise} we know that 
	there cannot be such an intersection of $\mathfrak{C}$ 
	with $(e^{-\upgamma(\mathsf{u}-\mathsf{f}_m(0))},+\infty)$
	or with $(-\infty,e^{-\upgamma(\mathsf{u}+1)})$
	(besides the original critical point $w_{\mathsf{cr}}=0$).
	From the above description of monotonicity of $\Im S_m(x+\mathbf{i}\epsilon)$
	it readily follows that a neighborhood 
	within $(e^{-\upgamma(\mathsf{u}+1)},e^{-\upgamma(\mathsf{u}-\mathsf{f}_m(0))})$
	where $\Im S_m(x)=0$ should be inside $(e^{-\upgamma},1)$,
	and therefore the point of intersection of $\mathfrak{C}$ with $\mathbb{R}$ 
	is also inside $(e^{-\upgamma},1)$.
	
	\begin{figure}[htpb]
		\centering
		\includegraphics[width=.6\textwidth]{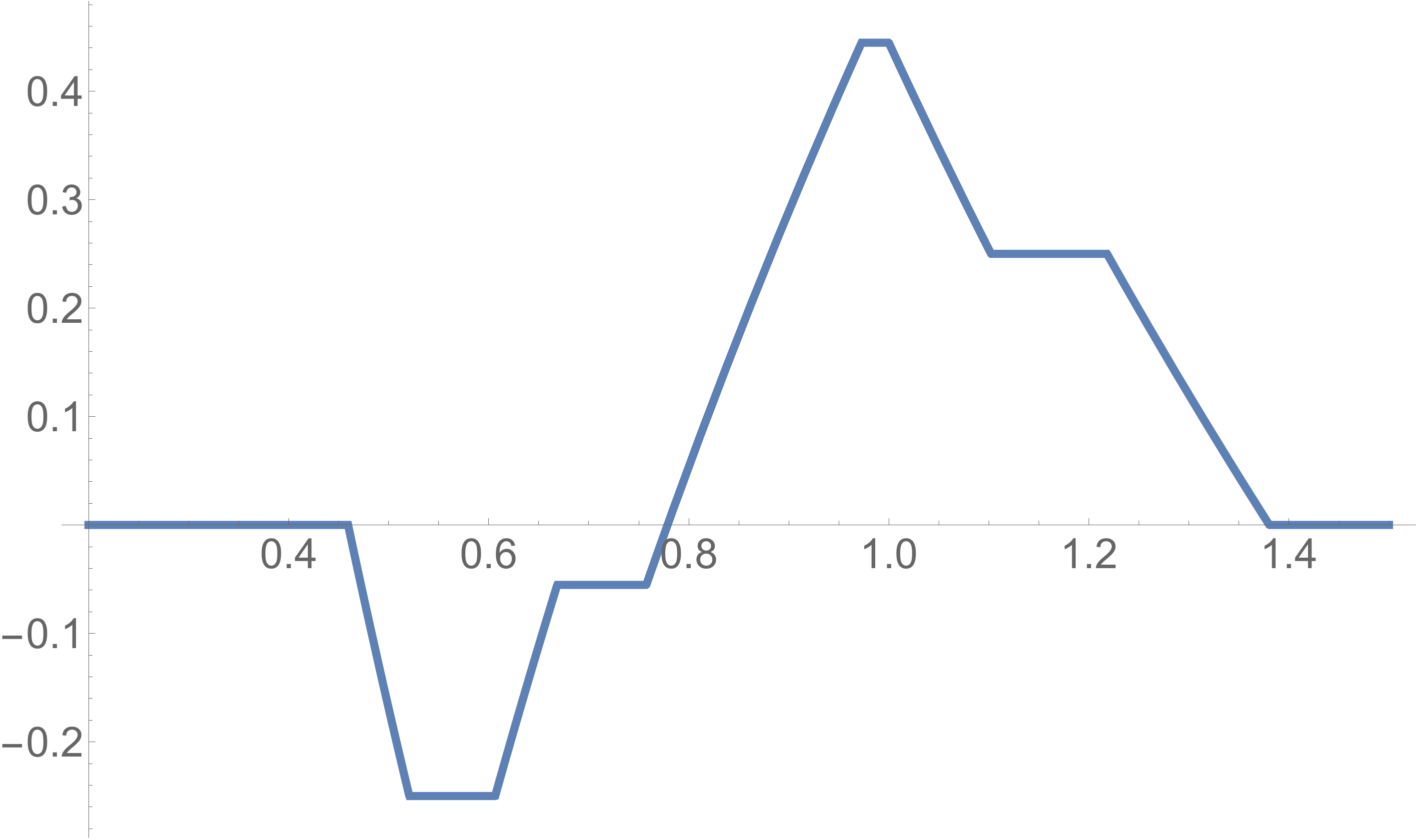}
		\caption{%
			The plot of $\frac{1}{\pi}\Im S(x+\mathbf{i}\epsilon)$
			for small $\epsilon>0$ and a piecewise constant function $\mathsf{f}=\mathsf{f}_m$. 
			The parameters are $m=4$,
			$\upgamma=\frac{1}{2}$ (so $e^{-\upgamma}\approx 0.606$),
			$\alpha=\left\{ \frac{6}{5}, 1, \frac{1}{2}, 0 \right\}$,
			and
			$s=\left\{ 0,\frac{1}{4},\frac{1}{2},\frac{3}{4},1 \right\}$.%
		}
		\label{fig:plot_Im}
	\end{figure}

	In the other case $\mathfrak{C}$ intersects $\mathbb{R}$ at a point which
	is not a critical point of $S_m$, and this corresponds to a transversal 
	intersection of the graph of $\Im S_m(x+\mathbf{i}\epsilon)$ with the $x$-axis
	(this case is shown in \Cref{fig:plot_Im}).
	The monotonicity of 
	$\Im S_m(x+\mathbf{i}\epsilon)$
	implies that such a transversal intersection can occur only inside $(e^{-\upgamma},1)$,
	which completes the proof in the case of 
	a piecewise constant function $\mathsf{f}=\mathsf{f}_m$.

	In the case of an arbitrary $\mathsf{f}$ we argue similarly to the proof of 
	\Cref{proposition:other_critical_points_general} and approximate 
	$\mathsf{f}$ by piecewise constant functions $\mathsf{f}_m$ as $m\to+\infty$.
	The segments on which $\Im S(x+\mathbf{i}\epsilon)$ is weakly monotone
	for small $\epsilon>0$ are the same as in the piecewise constant case, and so
	there is $x\in(e^{-\upgamma},1)$ where $\Im S(x+\mathbf{i}\epsilon)=0$.
	This $x$ corresponds to the intersection of the contour $\mathfrak{C}$ 
	with $\mathbb{R}$.
\end{proof}

\Cref{lemma:S_double_prime,lemma:behavior_at_infinity,lemma:intersect_real_line}
imply that the steepest descent contour $\mathfrak{C}$ looks as in \Cref{fig:steep_contour}.
Note that the numerator of the integrand in $A_i$ of \Cref{proposition:qA_i_best}
has zeros at $w=q^r$, $r=1,\ldots, N-K-1$, while the denominator
in this integrand does not have double zeros. 
This implies that the integrand in $A_i$ has no poles between
$q^{N-K-1}=e^{-\upgamma+\upgamma\frac{K+1}{N}}\sim e^{-\upgamma}$ and $q=e^{-\frac{\upgamma}{N}}\sim 1$.
Moreover, the integrand in $A_i$ also has no poles on the negative real line.
Therefore, the deformation of the integration contour $C'$ in formula
\eqref{qA_i_best} of \Cref{proposition:qA_i_best} to the steepest 
descent contour $\mathfrak{C}$ does not give rise to any residues.
In the sequel we will thus assume that the integration contour 
in \eqref{qA_i_best} is already~$\mathfrak{C}$.

\subsection{First level CLT}
\label{sub:level_1_CLT}

Before proving \Cref{thm:main_result} in full generality 
in \Cref{sec:completing_proof} below, 
let us use the results of \Cref{sub:critical_piecewise,sub:contour}
to establish the simpler particular case $K=1$ of this theorem.
That is, we are interested in establishing
a Central Limit Theorem for
the probability measure on $\mathbb{Z}$
given in \eqref{integral_formula_for_K1},
where the integration contour is $\mathfrak{C}$, the steepest descent one.
In the limit at $N\to+\infty$ the contribution to the integral outside a neighborhood of 
$w_{\mathsf{cr}}=0$ of size, say, $N^{-\frac{1}{6}}$, is negligible.
Indeed, outside this neighborhood 
$N\Re\left( S(w)-S(0) \right)$ is bounded from above by a negative constant times 
$N^{\frac{2}{3}}$, which dominates the $O(\sqrt N)$ terms in the exponent under the integral
(in, e.g., \eqref{O_sqrt_N_part}).
Therefore, we can focus on the behavior of the integral
over the part of the contour $\mathfrak{C}$
in a $N^{-\frac{1}{6}}$ neighborhood of $w_{\mathsf{cr}}=0$.

In this neighborhood of zero we change 
the variables as $w=-\mathbf{i}N^{-\frac{1}{2}}t$, $t\in \mathbb{R}$. 
The minus sign comes from the orientation of the contour $\mathfrak{C}$.
With this change of variables and with scaling $x=\mathsf{u}N+\xi N^{\frac{1}{2}}$,
$\xi\in \mathbb{R}$, \eqref{integral_formula_for_K1} becomes
\begin{equation}
\label{first_level_CLT_proof1}
	\begin{split}
		&P^{N}_{q,\nu}
		\left( \lambda^1_1=\mathsf{u}N+\xi N^{\frac{1}{2}}+1 \right)
		=
		-
		e^{
			-\frac{\upgamma}{2}-\upgamma\xi N^{\frac{1}{2}}-
			\upgamma N\left( \mathsf{u}-\frac{1}{2} \right)
		}
		\bigl(1-e^{-\upgamma}+O(1/N)\bigr)
		\\&\hspace{65pt}\times
		\frac{1}{2\pi\mathbf{i}}
		\int_{-N^{\frac{1}{3}}}^{N^{\frac{1}{3}}}
		\bigl(-\mathbf{i}N^{-\frac{1}{2}}dt\bigr)
		\frac
		{
			e^{-\frac{\upgamma (N-2) (N-1)}{2 N}}
			\prod_{r=1}^{N-2}(1+\mathbf{i}N^{-\frac{1}{2}}te^{\frac{\upgamma r}{N}})
		}
		{
			e^{-\upgamma \mathsf{u}N-\upgamma\xi N^{\frac{1}{2}}}
			\prod_{r=1}^N (
				1
				+
				\mathbf{i}N^{-\frac{1}{2}}te^{-\frac{\upgamma \nu_r}{N}+\frac{\upgamma r}{N}}
				e^{\upgamma \mathsf{u}+\upgamma\xi N^{-\frac{1}{2}}}
			)
		}.
	\end{split}
\end{equation}
Here we took the factors out of the products
in order to Taylor expand in the small variable $N^{-\frac{1}{2}}t$
for fixed $t\in\mathbb{R}$ (note that $r/N$ and $\nu_r/N$ stay bounded as $N\to+\infty$):
\begin{equation*}
	\begin{split}
		\prod_{r=1}^{N-2}(1+\mathbf{i}N^{-\frac{1}{2}}te^{\frac{\upgamma r}{N}})
		&=
		\exp
		\left\{  
			\sum_{r=1}^{N-2}
			\log
			\bigl(
				1+\mathbf{i}N^{-\frac{1}{2}}te^{\frac{\upgamma r}{N}}
			\bigr)
		\right\}
		\\
		&=
		\exp
		\left\{ 
			N^{-\frac{1}{2}}
			\sum_{r=1}^{N-2}
			\mathbf{i}te^{\frac{\upgamma r}{N}}
			+
			N^{-1}
			\sum_{r=1}^{N-2}
			\frac{t^2}{2}e^{\frac{2\upgamma r}{N}}
			+
			O(N^{-\frac{1}{2}})
		\right\}
		.
	\end{split}
\end{equation*}
In the denominator we also need to keep track of 
the term 
$e^{\upgamma \xi N^{-1/2}}=1+N^{-\frac{1}{2}}\upgamma\xi+O(N^{-1})$,
and so we have
\begin{equation*}
	\begin{split}
		&\prod_{r=1}^N (
			1
			+
			\mathbf{i}N^{-\frac{1}{2}}te^{-\frac{\upgamma \nu_r}{N}+\frac{\upgamma r}{N}}
			e^{\upgamma \mathsf{u}+\upgamma\xi N^{-\frac{1}{2}}}
		)
		=
		\exp
		\left\{ 
			\sum_{r=1}^N \log
			\bigl(
				1
				+
				\mathbf{i}N^{-\frac{1}{2}}t
				e^{\upgamma\left( \mathsf{u}+\frac{r}{N}-\frac{\nu_r}{N} \right)}
				e^{\upgamma\xi N^{-\frac{1}{2}}}
			\bigr)
		\right\}
		\\
		&\hspace{15pt}
		=
		\exp
		\left\{ 
			N^{-\frac{1}{2}}
			\sum_{r=1}^N
			\mathbf{i}t
			e^{\upgamma\left( \mathsf{u}+\frac{r}{N}-\frac{\nu_r}{N} \right)}
			+
			N^{-1}\sum_{r=1}^{N}
			\left( 
				\mathbf{i}t
				\upgamma\xi 
				e^{\upgamma\left( \mathsf{u}+\frac{r}{N}-\frac{\nu_r}{N} \right)}
				+
				\frac{t^2}{2}
				e^{2\upgamma\left( \mathsf{u}+\frac{r}{N}-\frac{\nu_r}{N} \right)}
			\right)
			+O(N^{-\frac{1}{2}})
		\right\}.
\end{split}
\end{equation*}
Every sum in the exponent in the previous two formulas
is a Riemann sum for the corresponding integral, and
combining all exponents in \eqref{first_level_CLT_proof1} together
we get an exponent of
\begin{align*}
	&
	-\frac{\upgamma}{2}-\upgamma\xi N^{\frac{1}{2}}-
	\upgamma N\left( \mathsf{u}-\frac{1}{2} \right)
	-\frac{\upgamma (N-2) (N-1)}{2 N}
	+\upgamma \mathsf{u}N+\upgamma\xi N^{\frac{1}{2}}
	+
	\mathbf{i}tN^{\frac{1}{2}}
	\int_0^1e^{\upgamma s}ds
	\\
	&\hspace{20pt}+
	\frac{t^2}{2}\int_0^1e^{2\upgamma s}ds
	-
	\mathbf{i}tN^{\frac{1}{2}}
	\int_0^1e^{\upgamma(\mathsf{u}+s-\mathsf{f}(s))}ds
	-\mathbf{i}t\upgamma\xi
	\int_0^1e^{\upgamma(\mathsf{u}+s-\mathsf{f}(s))}ds
	-
	\frac{t^2}{2}
	\int_0^1
	e^{2\upgamma(\mathsf{u}+s-\mathsf{f}(s))}ds
	\\&
	=
	\upgamma(1-N^{-1})
	-
	\mathbf{i}t\xi\left( e^\upgamma-1 \right)
	-
	\frac{t^2}{2}S''(0)
	,
\end{align*}
where we used the definition of $\mathsf{u}$
\eqref{u_correct_value} and the fact that
\begin{equation}
	\label{s_double_prime}
	S''(0)
	=
	-\int_0^1e^{2s\upgamma}ds
	+
	\int_0^1
	e^{2\upgamma(\mathsf{u}+s-\mathsf{f}(s))}ds
\end{equation}
(which immediately follows from \eqref{S_action_function_definition}).
Therefore, we can continue \eqref{first_level_CLT_proof1} 
and evaluate the Gaussian integral
\begin{align*}
	P^{N}_{q,\nu}
	\left( \lambda^1_1=\mathsf{u}N+\xi N^{\frac{1}{2}}+1 \right)
	&=
	\frac{\left( 1+o(1) \right)}{N^{\frac{1}{2}}}
	\frac{e^{\upgamma}-1}{2\pi}
	\int_{-\infty}^{+\infty}
	\exp
	\left\{ 
		-
		\mathbf{i}t\xi\left( e^\upgamma-1 \right)
		-
		\frac{t^2}{2}S''(0)
	\right\}
	dt
	\\&=
	\frac{\left( 1+o(1) \right)}{N^{\frac{1}{2}}}
	\frac{e^{\upgamma}-1}{\sqrt{2\pi S''(0)}}
	\exp\left\{ -\xi^2\frac{(e^{\upgamma}-1)^2}{2S''(0)} \right\}
	,
\end{align*}
which leads to a Gaussian density in the variable $\xi\in\mathbb{R}$
(the factor $N^{-\frac{1}{2}}$
corresponds to the rescaling of the space from discrete to continuous).
Thus, we have established the particular case $K=1$ of 
\Cref{thm:main_result}:
\begin{proposition}[First level CLT]
	\label{prop:first_level_CLT}
	Under Assumptions \ref{ass:function} and \ref{ass:q},
	as $N\to+\infty$, we have the following convergence in distribution
	of the location $\lambda^1_1=\lambda^1_1(N)$ of the vertical lozenge 
	on the first level:
	\begin{equation*}
		\frac{\lambda^1_1-N\mathsf{u}}{\sqrt N}\to
		\mathsf{L}^{1}_1\sim
		\mathsf{N}(0,\upsigma^2).
	\end{equation*}
	Here the global location $\mathsf{u}$ 
	and the limiting variance $\upsigma^2$
	are given by \eqref{u_intro}
	and \eqref{sigma_square}, respectively.
\end{proposition}

\section{Completing the proof of \Cref{thm:main_result}}
\label{sec:completing_proof}

In this section we finish the proof of our main result,
\Cref{thm:main_result}. 
We first consider the asymptotics of the distribution 
at each level $K=1,2,\ldots $ 
(%
	using the formulas of \Cref{thm:Pet14} and \Cref{proposition:qA_i_best} 
	together with our results 
	from \Cref{sec:asymptotic_mean,sec:critical_points}%
),
and then explain how one obtains a joint ICLT for all $K(K+1)/2$ lozenges
$\boldsymbol\lambda=\bigl\{ \lambda^{k}_j\colon k=1,\ldots,K,\,j=1,\ldots, k \bigr\}$.

\subsection{CLT at level $K$}
\label{sub:K_level_CLT}

Let us begin with the behavior of the prefactor in \eqref{formula_Pet14}:
\begin{lemma}
	\label{lemma:prefactor_behavior}
	Fix $K\ge1$. Let $\varkappa=(\varkappa_1\ge \ldots \ge \varkappa_K)$
	depend on $N$ as $\varkappa_j=\mathsf{u}N+\xi_j N^{\frac{1}{2}}$,
	where 
	$\mathsf{u}$ is given by \eqref{u_correct_value},
	and
	$\xi_1\ge \ldots\ge\xi_K $ are fixed real numbers. 
	Then as $N\to+\infty$ the prefactor in front of the determinant in 
	\eqref{formula_Pet14} behaves as
	\begin{multline*}
		s_{\varkappa}(1,q,\ldots,q^{K-1})\,
		(-1)^{K(N-K)}
		q^{(N-K)|\varkappa|}
		q^{-K(N-K)(N+2)/2}
		=
		\bigl(1+o(1)\bigr)
		(-1)^{K(N-K)}
		N^{\frac{K(K-1)}{4}}
		\\\times
		\exp
		\Bigl\{
			\tfrac{1}{2}\upgamma K(\mathsf{u}(K+1)-K+2)
			-\upgamma N^{\frac{1}{2}}\sum_{i=1}^{K}\xi_i
			-\upgamma K N\bigl(\mathsf{u}-\tfrac{1}{2}\bigr)
		\Bigr\}
		\prod_{1\le i<j\le K}
		\frac{
			\xi_i
			-\xi_j
		}
		{j-i}.
	\end{multline*}
\end{lemma}
\begin{proof}
	We have by \eqref{Schur_q_specialization}:
	\begin{align*}
		s_{\varkappa}(1,q,\ldots,q^{K-1} )
		&=
		\prod_{1\le i<j\le K}
		\frac{q^{\varkappa_i-i}-q^{\varkappa_j-j}}
		{q^{-i}-q^{-j}}
		\\&=
		\prod_{1\le i<j\le K}
		\frac{
			e^{-\upgamma\mathsf{u}-\frac{\upgamma\xi_i}{\sqrt N}+\frac{i \upgamma}{N}}
			-
			e^{-\upgamma\mathsf{u}-\frac{\upgamma\xi_j}{\sqrt N}+\frac{j \upgamma}{N}}
		}
		{e^{\frac{i \upgamma}{N}}-e^{\frac{j \upgamma}{N}}}
		\\&=
		e^{-\frac{1}{2}\upgamma\mathsf{u}K(K-1)}
		\prod_{1\le i<j\le K}
		\frac{
			-\frac{\upgamma\xi_i}{\sqrt N}
			+\frac{\upgamma\xi_j}{\sqrt N}
			+O(N^{-1})
		}
		{\frac{i \upgamma}{N}-\frac{j \upgamma}{N}+O(N^{-2})}
		\\
		&=
		\bigl(1+o(1)\bigr)
		e^{-\frac{1}{2}\upgamma\mathsf{u}K(K-1)}
		N^{\frac{K(K-1)}{4}}
		\prod_{1\le i<j\le K}
		\frac{
			\xi_i
			-\xi_j
		}
		{j-i}
		,
	\end{align*}
	and, moreover,
	$|\varkappa|=\mathsf{u}KN+N^{\frac{1}{2}}\sum_{i=1}^{K}\xi_i$.
	This implies the claim.
\end{proof}

Let us now address the asymptotic behavior of the individual 
entries $A_i(\varkappa_j-j)$ of the 
$K\times K$
determinant in \eqref{formula_Pet14}. 
The analysis is similar to \Cref{sub:level_1_CLT}, but some care is required to
pass from the asymptotics of the individual entries to the asymptotics of the
determinant. 
This dictates our formulation of \Cref{lemma:general_q_Ai_behavior} below.

We need to introduce some notation. 
Define $c_{i,l}$, $i=1,\ldots,K $, $l=0,\ldots,K-1 $, to be the 
coefficient of $y^l$ in the polynomial 
$(1+y)^{i-1}(1+e^{\upgamma}y)^{K-i}$. Also set
\begin{equation*}
	G_l(\xi):=\frac{e^{\upgamma}-1}{2\pi }\int_{-\infty}^{+\infty}
	\exp\left\{ 
		-\mathbf{i}t\xi(e^{\upgamma}-1)-\frac{t^2}{2}S''(0)
	\right\}
	(\mathbf{i}t)^l 
	dt
	,\qquad l=0,1,2,\ldots, 
\end{equation*}
where $S''(0)$ is given by \eqref{s_double_prime}.
\begin{lemma}
	\label{lemma:general_q_Ai_behavior}
	Fix $K\ge1$ and let
	$x=\mathsf{u}N+\xi N^{\frac{1}{2}}$. 
	Then $A_i(x)$ given by \eqref{qA_i_best} has the following asymptotics as $N\to+\infty$:
	\begin{multline}
		\label{general_qA_i_formulation}
		A_i(x)
		=
		\bigl(1+o(1)\bigr)
		(-1)^{N-K}
		N^{-\frac{1}{2}}
		\exp
		\left\{ 
			\upgamma N\bigr(\mathsf{u}-\tfrac{1}{2}\bigr)
			+\upgamma\xi N^{\frac{1}{2}}
			+\upgamma\bigl( -\tfrac{1}{2}+K-K\mathsf{u} \bigr)
		\right\}
		\\\times
		\sum_{l=0}^{K-1}c_{i,l}
		N^{-\frac{l}{2}}
		e^{\upgamma\mathsf{u}l}
		G_l(\xi)
		\bigl(1+\epsilon_{i,l,\xi}(N)\bigr),
	\end{multline}
	where the remainder $o(1)$ in front does not depend on $i$ while each
	$\epsilon_{i,l,\xi}(N)=o(1)$ may depend on $i$, $l$, and $\xi$.
\end{lemma}
\begin{proof}
	Arguing as in \Cref{sub:level_1_CLT}, we can restrict the
	integration to a small neighborhood of $w_{\mathsf{cr}}=0$, 
	and change the variables as $w=-\mathbf{i}tN^{-\frac{1}{2}}$, $t\in \mathbb{R}$.
	Then we have
	\begin{multline*}
		A_i(x)=
		\frac{q^{(K-N)x}q^{\frac{1}{2}(N-K)(N-K-1)}(-1)^{N-K-1}
		(1-q^{N-K})}{2\pi\mathbf{i}}
		\\
		\oint_{\mathfrak{C}}dw\,
		\frac{\prod_{r=1}^{N-K-1}(1-wq^{-r})}
		{\prod_{r=1}^N (1-wq^{-x+\nu_r-r})}
		\prod_{r\in I_{i,N,K}}(1-wq^{-x-r}),
	\end{multline*}
	where recall that
	$I_{i,N,K}=\left\{ 1,\ldots,i-1  \right\}\cup\left\{ N-K+i+1,\ldots,N  \right\}$.
	The prefactor scales as
	\begin{multline}
		\label{general_qA_i_proof0}
		(-1)^{N-K-1}
		q^{(K-N)x}q^{\frac{1}{2}(N-K)(N-K-1)}
		(1-q^{N-K})
		\\=
		\bigl(1+o(1)\bigr)
		(-1)^{N-K-1}
		\exp
		\left\{ 
			\upgamma N\bigr(\mathsf{u}-\tfrac{1}{2}\bigr)
			+\upgamma\xi N^{\frac{1}{2}}
			+\upgamma\bigl( \tfrac{1}{2}+K-K\mathsf{u} \bigr)
		\right\}
		(1-e^{-\upgamma}).
	\end{multline}
	The part of the integrand not depending on $i$ behaves as
	\begin{multline}
		\label{general_qA_i_proof1}
		\frac{\prod_{r=1}^{N-K-1}(1+\mathbf{i}tN^{-\frac{1}{2}}e^{\frac{\upgamma r}{N}})}
		{\prod_{r=1}^N (1+\mathbf{i}tN^{-\frac{1}{2}}
		e^{\upgamma\mathsf{u}+\upgamma\xi N^{-\frac{1}{2}}-\frac{\upgamma\nu_r}N+\frac{\upgamma r}{N}})}
		d(-\mathbf{i}tN^{-\frac{1}{2}})
		\\=
		-\bigl(1+o(1)\bigr)
		\mathbf{i}N^{-\frac{1}{2}}
		\exp\left\{ 
			-\mathbf{i}t\xi(e^{\upgamma}-1)-\frac{t^2}{2}S''(0)
		\right\}
		dt,
	\end{multline}
	where we employed Taylor expansions and approximation of sums by the corresponding 
	Riemann integrals similarly to \Cref{sub:level_1_CLT}
	(the difference by finitely many factors in the numerator does not affect the approximation).
	
	The product over $I_{i,N,K}$ contains only $K-1$ factors, which is finite
	and thus cannot be approximated by the exponent of an integral. 
	We have for $r\le i-1$:
	\begin{equation}
		\label{general_qA_i_proof2}
		1-wq^{-x-r}
		=
		1+\mathbf{i}tN^{-\frac{1}{2}}
		e^{\upgamma\mathsf{u}+\upgamma\xi N^{-\frac{1}{2}}+\frac{r\upgamma}{N}}
		=
		1+\mathbf{i}tN^{-\frac{1}{2}}
		e^{\upgamma\mathsf{u}}\bigl(1+O(N^{-\frac{1}{2}})\bigr).
	\end{equation}
	Similarly, for $r\ge N-K+i+1$:
	\begin{equation}
		\label{general_qA_i_proof3}
		1-wq^{-x-r}
		=
		1+\mathbf{i}tN^{-\frac{1}{2}}
		e^{\upgamma\mathsf{u}+\upgamma\xi N^{-\frac{1}{2}}+\frac{r\upgamma}{N}}
		=
		1+\mathbf{i}tN^{-\frac{1}{2}}
		e^{\upgamma\mathsf{u}}e^{\upgamma}\bigl(1+O(N^{-\frac{1}{2}})\bigr).
	\end{equation}
	Therefore, we have
	\begin{equation}\label{general_qA_i_proof4}
		\prod_{r\in I_{i,N,K}}(1-wq^{-x-r})
		=
		\sum_{l=0}^{K-1}c_{i,l}
		\bigl( \mathbf{i}tN^{-\frac{1}{2}}e^{\upgamma\mathsf{u}} \bigr)^l
		\bigl(1+\epsilon_{i,l,\xi}(N)\bigr),
	\end{equation}
	where $\epsilon_{i,l,\xi}(N)=O(N^{-\frac{1}{2}})$
	comes from the remainders in \eqref{general_qA_i_proof2}, 
	\eqref{general_qA_i_proof3}.
	Here we use a different notation to contrast with the remainders 
	$o(1)$ in \eqref{general_qA_i_proof0}, \eqref{general_qA_i_proof1}
	which are independent of $i$.
	Combining \eqref{general_qA_i_proof0}, \eqref{general_qA_i_proof1}, and \eqref{general_qA_i_proof4},
	we get the claim.
\end{proof}
\Cref{lemma:general_q_Ai_behavior} implies that the determinant 
$\det[ A_{i}(\varkappa_j-j)]$
of \Cref{thm:Pet14} is, up to a prefactor, 
asymptotically equal to a product of two 
simpler determinants, one of the $c_{i,l}$'s,
and the other one involving the $G_l(\xi_j)$'s
(where $\varkappa_j=\mathsf{u}N+\xi_j N^{\frac{1}{2}}$),
as long as both of these simpler determinants are nonzero.
Indeed, under the latter assumption in the expansion of
$\det\left[ \sum_{l=0}^{K-1}c_{i,l}G_l(\xi_j)(1+\epsilon_{i,l,\xi_j}(N)) \right]$
as a sum over permutations the asymptotically dominating terms
do not contain any of the $\epsilon_{i,l,\xi_j}$'s, and 
lead to a product of two nonzero determinants.
Let us now compute these two simpler determinants separately:
\begin{lemma}
	\label{lemma:det_c}
	We have 
	\begin{equation*}
		\Det_{1\le i\le K,\,0\le l\le K-1}
		[c_{i,l}]=
		(1-e^{\upgamma})^{\frac{K(K-1)}{2}}.
	\end{equation*}
\end{lemma}
\begin{proof}
	Introduce variables $y_1,\ldots,y_K $. 
	Then
	\begin{align*}
		\Det_{1\le i\le K,\,0\le l\le K-1}
		[c_{i,l}]&=
		\frac{1}{\prod_{i>j}(y_i-y_j)}
		\Det_{i,j=1}^{K}\left[ \sum_{l=0}^{K-1}c_{i,l}y_j^{l} \right]
		\\&=
		\frac{1}{\prod_{i>j}(y_i-y_j)}
		\Det_{i,j=1}^{K}\left[ (1+y_j)^{i-1}(1+e^{\upgamma} y_j)^{K-i}\right]
	\end{align*}
	by the very definition of $c_{i,l}$.
	The latter determinant can be rewritten as
	\begin{align*}
		\Det_{i,j=1}^{K}\left[ (1+y_j)^{i-1}(1+e^{\upgamma} y_j)^{K-i}\right]
		&=
		\prod_{j=1}^{K}(1+e^{\upgamma}y_j)^{K-1}
		\Det_{i,j=1}^{K}
		\left[ 
			\left( \frac{1+y_j}{1+e^{\upgamma}y_j} \right)^{i-1}
		\right]
		\\&=
		\prod_{j=1}^{K}(1+e^{\upgamma}y_j)^{K-1}
		\prod_{i<j}\left(
		\frac{1+y_i}{1+e^{\upgamma}y_i}
		-
		\frac{1+y_j}{1+e^{\upgamma}y_j}
		\right)
		\\&=
		(1-e^{\upgamma})^{\frac{K(K-1)}{2}}\prod_{i>j}(y_i-y_j),
	\end{align*}
	as desired.
\end{proof}
Recall the notation 
$\upsigma^2=\frac{S''(0)}{(e^{\upgamma}-1)^2}$, 
cf. \eqref{sigma_square},\eqref{s_double_prime}.
\begin{lemma}
	\label{lemma:det_G}
	Let $\xi_1\ge \ldots\ge \xi_K$, $\xi_j\in\mathbb{R}$. 
	We have
	\begin{equation*}
		\Det_{0\le l\le K-1,\,1\le j\le K}
		\left[ G_l(\xi_j) \right]=
		\frac{G_0(\xi_1)\ldots G_0(\xi_K)}{\upsigma^{K(K-1)}(e^{\upgamma}-1)^{\frac{K(K-1)}{2}}}
		\prod_{i>j}(\xi_i-\xi_j).
	\end{equation*}
\end{lemma}
\begin{proof}
	By evaluating the Gaussian integral, we have 
	\begin{equation*}
		G_0(\xi)=\frac{1}{\sqrt{2\pi\upsigma^2}}e^{-\frac{\xi^2}{2\upsigma^2}}.
	\end{equation*}
	Moreover, by integrating by parts one can readily check that
	\begin{equation*}
		G_l(\xi)=\frac{\xi}{\upsigma^2(e^{\upgamma}-1)}G_{l-1}(\xi)-\frac{l-1}{\upsigma^2(e^{\upgamma}-1)^2}G_{l-2}(\xi),
	\end{equation*}
	which implies that $G_l(\xi)/G_0(\xi)$ is a 
	polynomial\footnote{%
		In fact, $G_l(\xi)$, $l=0,1,\ldots$, can be 
		expressed through the classical Hermite orthogonal polynomials, 
		but we do not need this precise form of the $G_l$'s.%
	}
	of degree $l$ in $\xi$ with leading coefficient equal to
	$(e^{\upgamma}-1)^{-l}\upsigma^{-2l}$.
	This implies the claim.
\end{proof}
\Cref{lemma:det_c,lemma:det_G} show that both 
determinants whose product enters the determinant of $A_i(\varkappa_j-j)$
in \eqref{formula_Pet14} are indeed nonzero
as long as the $\xi_j$'s are distinct. 
When some of the $\xi_j$'s are equal to each other, the GUE eigenvalue density also vanishes, 
and one can readily check that the scaling limit of \eqref{formula_Pet14} is zero, as it should be.
Therefore, we arrive at the following result:
\begin{proposition}(Level $K$ CLT)
	\label{prop:level_K_CLT}
	Under Assumptions \ref{ass:function} and \ref{ass:q},
	as $N\to+\infty$, for the locations 
	$\lambda^K_j=\lambda^K_j(N)$, $j=1,\ldots,K $,
	of the vertical lozenges at level $K$ we have the following convergence in distribution
	on $\mathbb{R}^K$:
	\begin{equation*}
		\biggl\{ 
			\frac{\lambda^K_j-N\mathsf{u}}{\sqrt N}
		\biggr\}_{j=1}^{K}
		\to 
		\left\{
			\mathsf{L}^K_j
		\right\}_{j=1}^{K}
		\sim \mathsf{GUE}_K(\upsigma^2).
	\end{equation*}
	Here $\mathsf{u}$ and $\upsigma^2$ are given by \eqref{u_intro} and \eqref{sigma_square}, respectively,
	and $\mathsf{GUE}_K$ is the eigenvalue distribution of the $K\times K$
	Gaussian Unitary Ensemble (cf. \Cref{def:GUE}).
\end{proposition}
\begin{proof}
	From \Cref{thm:Pet14} and \Cref{lemma:det_G,lemma:det_c,lemma:general_q_Ai_behavior,lemma:prefactor_behavior}
	we have after simplifications
	(%
		here, as before, $\varkappa_j=\mathsf{u}N+\xi_jN^{\frac{1}{2}}$ for
		$j=1,\ldots, K$%
	):
	\begin{multline}
		\label{level_K_CLT_proof}
		P^{N}_{q,\nu}\left( \lambda^K_j=\varkappa_j,\;j=1,\ldots,K  \right)
		=
		\bigl(1+o(1)\bigr)
		\frac{N^{-\frac{K}{2}}}{0!1!\ldots(K-1)!\upsigma^{K(K-1)}}
		\\\times
		G_0(\xi_1)\ldots G_0(\xi_K)
		\prod_{1\le i<j\le K}
		(\xi_i-\xi_j)^2.
	\end{multline}
	Here we also used the fact that the $K\times K$
	determinant in \eqref{thm:Pet14} has the form $\det [A_i(\varkappa_j-j)]$ 
	with the shifts in the $\varkappa_j$'s. 
	It can be seen from \eqref{proposition:qA_i_best} that 
	these shifts lead to an extra prefactor 
	$\prod_{j=1}^{K}q^{(N-K)j}=e^{-\frac{1}{2}\upgamma K(K+1)}(1+o(1))$.

	Because the prefactor $N^{-\frac{K}{2}}$ corresponds
	to the rescaling of the space from discrete to continuous, 
	we see that 
	the right-hand side 
	of \eqref{level_K_CLT_proof} 
	leads 
	to the eigenvalue density $\mathsf{GUE}_K(\upsigma^2)$.
	This completes the proof.
\end{proof}

\subsection{Gibbs property}
\label{sub:Gibbs}

The last step required to complete the proof of
\Cref{thm:main_result} is to show that the joint distribution
of all $K(K+1)/2$ locations $\boldsymbol\lambda=\{\lambda^{k}_{j}\colon k=1,\ldots,K,\,j=1,\ldots,k  \}$
of the vertical lozenges 
$\begin{tikzpicture}
	[scale=.17,very thick]
	\def\rt{0.866025}
	\lozv{(0,0)}{white}
\end{tikzpicture}$
at levels $1,\ldots, K$
is approximated by the GUE corners 
process $\mathbf{L}=\{ \mathsf{L}^{k}_{j}\}\sim
\mathsf{GUE}_{K\times K}(\upsigma^2)$
(see \Cref{def:GUE} for the latter).
We employ the Gibbs property of 
$\mathsf{GUE}_{K\times K}(\upsigma^2)$:
conditioned on the fixed $K$th row $\mathsf{L}^K$,
the 
distribution of the rest of the array
$\{\mathsf{L}^{k}_{j} \colon 1\le k\le K-1,\, 1\le j\le k \}$
is uniform over the polytope\footnote{%
	Often referred to as the Gelfand--Tsetlin polytope. 
	Note that it depends 
	on the fixed top row $\mathsf{L}^K$.%
}
described by the interlacing conditions
$\mathsf{L}^{k}_{j}\le\mathsf{L}^{k-1}_{j-1}\le\mathsf{L}^{k}_{j-1}$ 
for all 
$k=1,\ldots,K$
and 
$j=2,\ldots,k$.

At the same time, the pre-limit joint distribution
of the locations of the vertical lozenges
$\boldsymbol\lambda=\{\lambda^{k}_{j}\}$
satisfies a $q$-deformation of the Gibbs property.
Under the latter, 
the uniform conditional probabilities of 
the rows $\lambda^1,\ldots,\lambda^{K-1}$ conditioned on the fixed row $\lambda^K$
are replaced by the probabilities proportional to $q^{\mathsf{vol}}$.
Let us explain why in our limit regime 
the $q$-Gibbs property leads to the Gibbs one.

In \Cref{prop:level_K_CLT} we showed that the $K$th row
$\lambda^K$
of the discrete interlacing array corresponding to the vertical lozenges
in our $q^{\mathsf{vol}}$-distributed lozenge tiling 
converges after rescaling to the $K$th row $\mathsf{L}^K$
of the GUE corners process.
This rescaling in particular implies that the distances between 
the $\lambda^K_j$'s are of order $\sqrt N$.
Therefore, conditioned on the fixed configuration
$\lambda^K$ on the $K$th row 
with $\lambda^K_1-\lambda^K_K$ of order $\sqrt N$,
the difference between the maximal and the minimal volume of the 
tiling is also of order $\sqrt N$.
Thus, since $q=e^{-\upgamma/N}$, we have $q^{\max(\mathsf{vol})}=(1+o(1))q^{\min(\mathsf{vol})}$.
This means that the conditional distribution of the lower rows
$\lambda^1,\ldots,\lambda^{K-1} $
becomes uniform (up to the interlacing constraints) in the $N\to+\infty$ limit. 
This immediately leads to the ICLT and thus completes the proof of \Cref{thm:main_result}.

\printbibliography
\end{document}